\documentclass[12pt,letterpaper]{amsart}

\usepackage{epsfig,amsmath,amsthm,amssymb,graphicx}
\usepackage[top=1in, bottom=1in, left=1in, right=1in]{geometry}
\usepackage{color}
\usepackage{mathtools}
\usepackage{float}
\usepackage{tikz,environ}
\usepackage{epsfig}
\usepackage{caption}
\usepackage{subcaption}
\usetikzlibrary{decorations.pathmorphing}
\usetikzlibrary{intersections}
\usetikzlibrary{matrix,arrows,shapes}

\allowdisplaybreaks

\usepackage[alphabetic]{amsrefs}

\usepackage{pb-diagram}
	
\newtheorem{proposition}{Proposition}[section]
\newtheorem{theorem}[proposition]{Theorem}
\newtheorem{corollary}[proposition]{Corollary}

\newtheorem{example}[proposition]{Example}

\newtheorem*{theorem*}{Theorem}
\newtheorem*{proposition*}{Proposition}
\newtheorem*{lemma*}{Lemma}
\newtheorem*{corollary*}{Corollary}

\newtheorem*{rep@theorem}{\rep@title}
\newcommand{\newreptheorem}[2]{
\newenvironment{rep#1}[1]{
 \def\rep@title{#2 \ref{##1}}
 \begin{rep@theorem}}
 {\end{rep@theorem}}}
\makeatother

\theoremstyle{definition}
\newtheorem{definition}[proposition]{Definition}

\newreptheorem{theorem}{Theorem}
\newreptheorem{lemma}{Lemma}
\newreptheorem{proposition}{Proposition}
\newreptheorem{corollary}{Corollary}

\newcommand{\V}{\mathbb{V}}

\newcommand{\x}{\times}

\newcommand\myoverset[2]{\overset{\textstyle #1\mathstrut}{#2}}

\newcommand{\meq}[2][0mm]{\raisebox{-.88mm}{\hspace{#1}\ensuremath{#2}\hspace{#1}}} 

\newcommand{\tile}[2][1]{\begin{mosaic}[#1] #2 \\ \end{mosaic}}

\NewEnviron{mosaic}[1][1]{\begin{tikzpicture}[baseline]\def\tilesize{#1}\matrix[mosai,ampersand replacement=\&]{\BODY};\end{tikzpicture}}

\NewEnviron{vmosaic}[7][1]{\begin{tikzpicture}[baseline]\def\tilesize{#1}\matrix[mosai,ampersand replacement=\&]{\BODY};
   \foreach \x [count=\y] in #4
      \node[above,yshift=#1*.5cm-1] at (-#1*.5-#1*#3*.5+\y*#1,#1*#2*.5-#1*.5){$\mathstrut\x$};
   \foreach \x [count=\y] in #5
      \node[right,xshift=#1*.5cm+1] at (#1*#3*.5-#1*.5,#1*.5+#1*#2*.5-\y*#1){$\x$};
   \foreach \x [count=\y] in #6
      \node[below,yshift=-#1*.5cm] at (#1*.5+#1*#3*.5-\y*#1,-#1*#2*.5+#1*.5){$\mathstrut\x$};
   \foreach \x [count=\y] in #7
      \node[left,xshift=-#1*.5cm-1] at (-#1*#3*.5+#1*.5,-#1*.5-#1*#2*.5+\y*#1){$\x$};
   \end{tikzpicture}}

\newcommand{\cspace}{.1} 
\newcommand{\tileedge}{.05pt} 
\newcommand{\knotthickness}{1.5pt} 

\newcommand{\tileo}{\tikz \draw[tilestyle] (0,0) rectangle (\tilesize,\tilesize);}
\newcommand{\tilei}[1][{}]{\tikz {\draw[tilestyle] (0,0) rectangle (\tilesize,\tilesize); \draw[#1,knotstyle={\tilesize*\knotthickness}] (0,.5*\tilesize) to[out=0, in=90] (.5*\tilesize,0);}}
\newcommand{\tileii}[1][{}]{\tikz {\draw[tilestyle] (0,0) rectangle (\tilesize,\tilesize); \draw[#1,knotstyle={\tilesize*\knotthickness}] (\tilesize,.5*\tilesize) to[out=180, in=90] (.5*\tilesize,0);}}
\newcommand{\tileiii}[1][{}]{\tikz {\draw[tilestyle] (0,0) rectangle (\tilesize,\tilesize); \draw[#1,knotstyle={\tilesize*\knotthickness}] (\tilesize,.5*\tilesize) to[out=180, in=-90] (.5*\tilesize,\tilesize);}}
\newcommand{\tileiv}[1][{}]{\tikz {\draw[tilestyle] (0,0) rectangle (\tilesize,\tilesize); \draw[#1,knotstyle={\tilesize*\knotthickness}] (0,.5*\tilesize) to[out=0, in=-90] (.5*\tilesize,\tilesize);}}
\newcommand{\tilev}[1][{}]{\tikz {\draw[tilestyle] (0,0) rectangle (\tilesize,\tilesize); \draw[#1,knotstyle={\tilesize*\knotthickness}] (0,.5*\tilesize) to (\tilesize,.5*\tilesize);}}
\newcommand{\tilevi}[1][{}]{\tikz {\draw[tilestyle] (0,0) rectangle (\tilesize,\tilesize); \draw[#1,knotstyle={\tilesize*\knotthickness}] (.5*\tilesize,0) to (.5*\tilesize,\tilesize);}}
\newcommand{\tilevii}[1][{}]{\tikz {\draw[tilestyle] (0,0) rectangle (\tilesize,\tilesize); \draw[#1,knotstyle={\tilesize*\knotthickness}] (0,.5*\tilesize) to[out=0, in=90] (.5*\tilesize,0); \draw[#1,knotstyle={\tilesize*\knotthickness}] (\tilesize,.5*\tilesize) to[out=180, in=-90] (.5*\tilesize,\tilesize);}}
\newcommand{\tileviii}[1][{}]{\tikz {\draw[tilestyle] (0,0) rectangle (\tilesize,\tilesize); \draw[#1,knotstyle={\tilesize*\knotthickness}] (\tilesize,.5*\tilesize) to[out=180, in=90] (.5*\tilesize,0); \draw[#1,knotstyle={\tilesize*\knotthickness}] (0,.5*\tilesize) to[out=0, in=-90] (.5*\tilesize,\tilesize);}}
\newcommand{\tileix}[1][{}]{\tikz {\draw[tilestyle] (0,0) rectangle (\tilesize,\tilesize); \draw[#1,knotstyle={\tilesize*\knotthickness}] (0,.5*\tilesize) to (\tilesize,.5*\tilesize); \draw[#1,knotstyle={\tilesize*\knotthickness}] (.5*\tilesize,0) to (.5*\tilesize,.5*\tilesize-\tilesize*\cspace); \draw[#1,knotstyle={\tilesize*\knotthickness}] (.5*\tilesize,.5*\tilesize+\tilesize*\cspace) to (.5*\tilesize,\tilesize);}}
\newcommand{\tilex}[1][{}]{\tikz {\draw[tilestyle] (0,0) rectangle (\tilesize,\tilesize); \draw[#1,knotstyle={\tilesize*\knotthickness}] (.5*\tilesize,0) to (.5*\tilesize,\tilesize); \draw[#1,knotstyle={\tilesize*\knotthickness}] (0,.5*\tilesize) to (.5*\tilesize-\tilesize*\cspace,.5*\tilesize); \draw[#1,knotstyle={\tilesize*\knotthickness}] (.5*\tilesize+\tilesize*\cspace,.5*\tilesize) to (\tilesize,.5*\tilesize);}}

\tikzset{tilestyle/.style={use as bounding box, line width=\tileedge}} 
\tikzset{knotstyle/.style={line width=#1}} 
\tikzset{mosai/.style={execute at begin cell=\node\bgroup, execute at end cell=\egroup;, column sep=0mm, inner sep=0}}

\newcommand{\tilefit}{.7} 

\newcommand{\tilepix}[1][{}]{\tikz {\draw[tilestyle] (0,0) rectangle (\tilesize,\tilesize); \draw[#1->,knotstyle={\tilesize*\knotthickness}] (0,.5*\tilesize) to (\tilesize,.5*\tilesize); \draw[#1,knotstyle={\tilesize*\knotthickness}] (.5*\tilesize,0) to (.5*\tilesize,.5*\tilesize-\tilesize*\cspace); \draw[#1->,knotstyle={\tilesize*\knotthickness}] (.5*\tilesize,.5*\tilesize+\tilesize*\cspace) to (.5*\tilesize,\tilesize);}}

\newcommand{\tilenix}[1][{}]{\tikz {\draw[tilestyle] (0,0) rectangle (\tilesize,\tilesize); \draw[#1->,knotstyle={\tilesize*\knotthickness}] (0,.5*\tilesize) to (\tilesize,.5*\tilesize); \draw[#1->,knotstyle={\tilesize*\knotthickness}] (.5*\tilesize,.5*\tilesize-\tilesize*\cspace) to (.5*\tilesize,0) ; \draw[#1,knotstyle={\tilesize*\knotthickness}] (.5*\tilesize,.5*\tilesize+\tilesize*\cspace) to (.5*\tilesize,\tilesize);}}

\newcommand{\tilepx}[1][{}]{\tikz {\draw[tilestyle] (0,0) rectangle (\tilesize,\tilesize); \draw[#1->,knotstyle={\tilesize*\knotthickness}] (.5*\tilesize,\tilesize) to (.5*\tilesize,0); 
\draw[#1,knotstyle={\tilesize*\knotthickness}] (0,.5*\tilesize) to (.5*\tilesize-\tilesize*\cspace,.5*\tilesize);
\draw[#1->,knotstyle={\tilesize*\knotthickness}] (.5*\tilesize+\tilesize*\cspace,.5*\tilesize) to (\tilesize,.5*\tilesize);}}

\newcommand{\tilenx}[1][{}]{\tikz {\draw[tilestyle] (0,0) rectangle (\tilesize,\tilesize); \draw[#1->,knotstyle={\tilesize*\knotthickness}] (.5*\tilesize,0) to (.5*\tilesize,\tilesize); 
\draw[#1,knotstyle={\tilesize*\knotthickness}] (0,.5*\tilesize) to (.5*\tilesize-\tilesize*\cspace,.5*\tilesize);
\draw[#1->,knotstyle={\tilesize*\knotthickness}] (.5*\tilesize+\tilesize*\cspace,.5*\tilesize) to (\tilesize,.5*\tilesize);}}

\newcommand{\tileviio}[1][{}]{\tikz {\draw[tilestyle] (0,0) rectangle (\tilesize,\tilesize); \draw[#1->,knotstyle={\tilesize*\knotthickness}] (0,.5*\tilesize) to[out=0, in=90] (.5*\tilesize,0); \draw[#1->,knotstyle={\tilesize*\knotthickness}] (.5*\tilesize,\tilesize) to[out=-90, in=180] (\tilesize,.5*\tilesize);}}

\newcommand{\tileviiio}[1][{}]{\tikz {\draw[tilestyle] (0,0) rectangle (\tilesize,\tilesize); \draw[#1->,knotstyle={\tilesize*\knotthickness}] (.5*\tilesize,0) to[out=90, in=180] (\tilesize,.5*\tilesize); \draw[#1->,knotstyle={\tilesize*\knotthickness}] (0,.5*\tilesize) to[out=0, in=-90] (.5*\tilesize,\tilesize);}}

\newcommand{\tileviioa}[1][{}]{\tikz {\draw[tilestyle] (0,0) rectangle (\tilesize,\tilesize); \draw[red, ->,knotstyle={\tilesize*\knotthickness}] (0,.5*\tilesize) to[out=0, in=90] (.5*\tilesize,0); \draw[blue,->,knotstyle={\tilesize*\knotthickness}] (.5*\tilesize,\tilesize) to[out=-90, in=180] (\tilesize,.5*\tilesize);}}

\newcommand{\tilenxa}[1][{}]{\tikz {\draw[tilestyle] (0,0) rectangle (\tilesize,\tilesize); \draw[red,->,knotstyle={\tilesize*\knotthickness}] (.5*\tilesize,0) to (.5*\tilesize,\tilesize); 
\draw[blue,knotstyle={\tilesize*\knotthickness}] (0,.5*\tilesize) to (.5*\tilesize-\tilesize*\cspace,.5*\tilesize);
\draw[blue,->,knotstyle={\tilesize*\knotthickness}] (.5*\tilesize+\tilesize*\cspace,.5*\tilesize) to (\tilesize,.5*\tilesize);}}

\begin{document}
\title{Rectangular Mosaics for Virtual Knots}

\author{Taylor Martin}
\address{Department of Mathematics, Sam Houston State University}
\email{taylor.martin@shsu.edu}

\author{Rachel Meyers}
\address{Department of Mathematics, Louisiana State University}
\email{rmeye23@lsu.edu}


\begin{abstract} 
Mosaic knots, first introduced in 2008 by Lomanoco and Kauffman, have become a useful tool for studying combinatorial invariants of knots and links. In 2020, by considering knot mosaics on $n \times n$ polygons with boundary edge identification, Ganzell and Henrich extended the study of mosaic knots to include virtual knots - knots embedded in thickened surfaces. They also provided a set of virtual mosaic moves preserving knot and link type. In this paper, we introduce rectangular mosaics for virtual knots, defined to be $m \times n$ arrays of classical knot mosaic tiles, along with an edge identification of the boundary of the mosaic, whose closures produce virtual knots. We modify Ganzell and Henrich's mosaic moves to the rectangular setting, provide several invariants of virtual rectangular mosaics, and give algorithms for computations of common virtual knot invariants.  
\end{abstract}

\maketitle

\section{Introduction}

Mosaic knots, first introduced in 2008 by Lomanoco and Kauffman \cite{LK}, seek to lend structure to the study of knot theory and were initially introduced as a means of developing a quantum knot system. Subsequent work of Kuriya and Shehab \cite{KS} verified that mosaic knot theory is equivalent to classical knot theory. Mosaic knot theory constructs knot projections out of \emph{suitably connected tiles} by arranging the tiles depicted in Figure \ref{tiles} into a square array in such a way that the resulting grid contains a knot or link projection. Knot mosaics have become a rich arena for student research in knot theory. Many of the results focus on producing efficient mosaics by minimizing \emph{mosaic number}, the minimal $n$ such that a given knot $K$ can be represented on an $n \x n$ mosaic, or the \emph{tile number}, the minimal number of non-blank tiles needed to represent a given knot as a mosaic \cite{Ludwig,Lee}. For low crossing numbers, efficient mosaics have been tabulated \cite{Heap1, Heap2}. Others have explored hexagonal tiles, tiles with corner crossings, and mosaics for graph diagrams.  \cite{how, Choi, Heap3}   

\begin{figure}[htp]
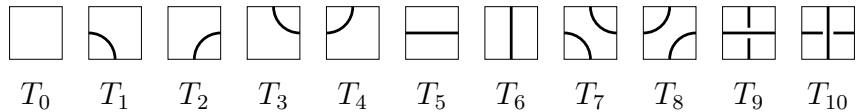

\centering
\setcounter{MaxMatrixCols}{20}$\begin{matrix} 
    \tile[\tilefit]{\tileo} & \tile[\tilefit]{\tilei} & \tile[\tilefit]{\tileii} & \tile[\tilefit]{\tileiii} & \tile[\tilefit]{\tileiv} & \tile[\tilefit]{\tilev} & \tile[\tilefit]{\tilevi} & \tile[\tilefit]{\tilevii} & \tile[\tilefit]{\tileviii} & \tile[\tilefit]{\tileix} & \tile[\tilefit]{\tilex} \\[2.5ex]
    T_0 & T_1 & T_2 & T_3 & T_4 & T_5 & T_6 & T_7 & T_8 & T_9 & T_{10}
\end{matrix}$\captionsetup{labelfont={sc}}
\caption{Mosaic tiles.}
\label{tiles}
\end{figure}

In 2020, Ganzell and Henrich \cite{GH} extended the study of knot mosaics to include virtual knots. By introducing an edge pairing on the $4n$ boundary edges of an $n \x n$ knot mosaic, one can produce a knot diagram $D$ on a closed, orientable surface $\Sigma_D$; this is one definition of a virtual knot. We refer to the quotient of the knot mosaic modulo its boundary edge identification as the {\it closure} of the virtual mosaic. The \emph{genus} of the diagram $D$ is defined to be the genus of the closed surface $\Sigma_D$. 

Note that this definition allows for tiles in a virtual knot mosaic to extend ``off the grid." If a virtual knot mosaic diagram contains two arcs that each extend across the boundary of the mosaic and if the edge labeling corresponding to those arcs form am \emph{interlocking pair} of boundary edges, as pictured in Figure \ref{crosspair}, the corresponding virtual knot will contain a \emph{virtual crossing}. 

The genus of a virtual mosaic can be determined by applying the classification of surfaces, and in particular, by counting the number of distinct vertices $v_D$ of the mosaic as created by the boundary identification. For an $n \x n$ virtual mosaic for a diagram $D$ with $v_D$ vertices,  $$g(D) = \frac{1 - v_D + 2n}{2}.$$ Then, the \emph{genus} of a virtual knot is defined to be the minimal genus $g(D)$ across any virtual mosaic $D$ representing $K$.  

\begin{figure}[H]
    \begin{tikzpicture}
    \hspace{-1.1in}
  $\alpha_i \cdots \alpha_j \cdots \alpha_i \cdots \alpha_j$
  \draw[line width=\tileedge] (-2,.45) to[out=90, in=90, looseness=1.4] (-.25,.45);
  \draw[line width=\tileedge] (-3,.45) to[out=90, in=90, looseness=1.4] (-1.25,.45);
  \draw[thin] (-1.62,1.05) circle (1.4mm);
\end{tikzpicture}
\vspace{.4 in}
\caption{An interlocking set of boundary edge pairs creates a virtual crossing}
    \label{crosspair}
\end{figure}
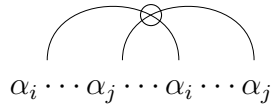

\begin{example}
    Figure \ref{tref}, from \cite{GH} shows a $2 \x 2 $ virtual mosaic with a boundary edge pairing, read counterclockwise from the northwest corner of the grid as the word $dcdcabba$ with an interlocking pair $dcdc$. There are 3 distinct vertices in the mosaic and the closure has genus 1. 

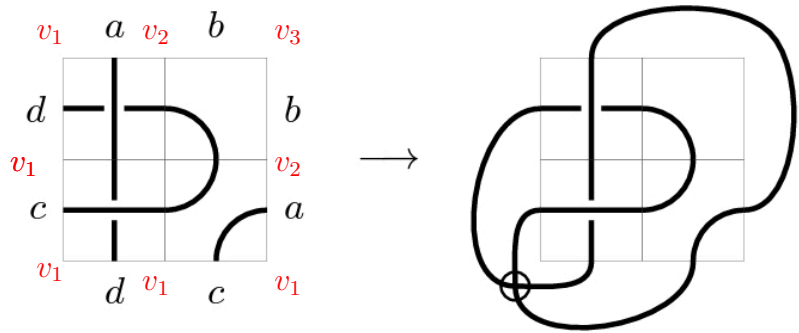
\begin{figure}[h]
\centering
\includegraphics[scale=.9]{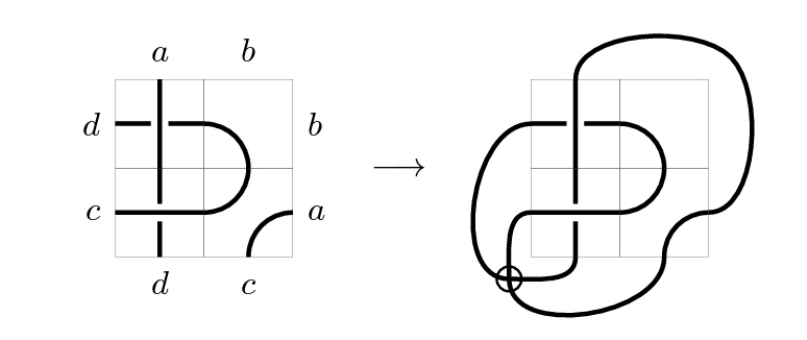}
\put(-310,120){\textcolor{red}{$v_1$}}
\put(-320,70){\textcolor{red}{$v_1$}}
\put(-310,30){\textcolor{red}{$v_1$}}
\put(-270,120){\textcolor{red}{$v_2$}}
\put(-270,25){\textcolor{red}{$v_1$}}
\put(-320,70){\textcolor{red}{$v_1$}}
\put(-220,70){\textcolor{red}{$v_2$}}
\put(-220,25){\textcolor{red}{$v_1$}}
\put(-220,120){\textcolor{red}{$v_3$}}

\caption{A $2 \x 2$ virtual mosaic with an edge pairing and its virtual knot closure.}
\label{tref}
\end{figure}
\end{example}
In this paper, we extend the study of virtual mosaic knot theory by introducing the {\it rectangular mosaic}. By considering rectangular mosaics, we can produce more space-efficient virtual mosaics for virtual and classical knots. 

\begin{definition}
    A {\it rectangular virtual knot mosaic} is an $m \x n$ array of mosaic tiles, along with an edge pairing of the $2(m+n)$ boundary edges of the array, such that the closure of the array yields a virtual knot diagram $D$ on a closed, orientable surface $\Sigma_D$. The \emph{genus} of the mosaic is defined to be the genus of $\Sigma_D$ and denoted $g(D)$. 
\end{definition}

Again, the genus of a rectangular $n \x m$ virtual knot mosaic $D$ can be determined by counting the number of vertices $v_D$ coming from the edge pairing of the mosaic and appealing to the classification of surfaces. $$g(D) = \frac{1-v_D + n + m}{2}.$$

In section \ref{sec:2}, we adapt Ganzell and Henrich's set of {\it virtual mosaic moves} to the new setting of rectangular mosaics, providing a modification for their injection move to both correct an error and adapt these moves to the rectangular case. We introduce new rectangular mosaic moves called \emph{row and column injection and ejection} that preserve virtual knot type. 

In section \ref{sec:3}, we introduce the \emph{tile number} of a virtual knot and define the {\it virtual row mosaic}, a $1 \x n$ virtual rectangular mosaic for a virtual knot or link. We prove that every virtual knot or link has a virtual row mosaic. We define the \emph{row number} invariant for a virtual knot and we prove results relating these invariants of virtual rectangular mosaics. 

In section \ref{sec:4}, we discuss virtual knot polynomials in the setting of virtual row mosaics. In section \ref{sec:5} we pose several questions for future study. We conclude with three tables. Table 1 includes all classical knots up to 8 crossings presented in mosaics realizing tile numbers. Table 2 depicts all virtual knots with up to 4 classical crossings realizing their tile numbers. Table 3 shows all alternating knots of up to 8 crossings presented on $2 \x 3 $ or $2 \x 4$ rectangular mosaics.

\section{Rectangular Mosaics and Virtual Mosaic Moves}\label{sec:2}

Much of the study of mosaic knots involves computing the \emph{mosaic number} $m(K)$ of a knot; given a knot $K$, find the minimal $n$ such that $K$ can be represented on an $n\times n$ mosaic. Ganzell and Henrich \cite{GH} focus on the \emph{virtual mosaic number} $m_v(K)$ of a knot, which is the minimal value of $n$ for which $K$ can be represented on an $n \times n$ virtual mosaic, and they show that for $K$ a classical knot, $m_v(K) \le m(K)-2$. By extending the definition of a virtual mosaic to allow for non-square tilings, we show that the total number of tiles can, in many cases, be further reduced. 

\begin{definition}
A \emph{virtual rectangular mosaic} is an $m \x n$ array of standard mosaic tiles, along with a pairing of the $2m+2n$ edges of the boundary of the array such that the closure of the mosaic produces a knot or link diagram on a compact, orientable surface.
\end{definition}


\begin{example}
Figure \ref{fig:Vmosaic Closure} shows an example of the $7_1$ knot as a $1 \x 7$ virtual mosaic as well as its closure. 
\begin{center}
\begin{figure}
    \begin{vmosaic}[1]{1}{7}{{d,c,c,b,b,a,a}}{{d}}{{e,h,h,g,g,f,f}}{{e}} 
    \tileix \& \tilex \& \tileix \& \tilex \& \tileix \& \tilex \& \tileix \\
    \end{vmosaic}
\vspace{.3 in}\\
\begin{tikzpicture}[baseline]
\hspace{-2 in}
\begin{vmosaic}[1]{1}{7}{{}}{{}}{{}}{{}} 
    \tileix \& \tilex \& \tileix \& \tilex \& \tileix \& \tilex \& \tileix \\
    \end{vmosaic}

\draw[line width=\knotthickness] (-5.5,.45) to[out=90, in=90, looseness=1.4] (-4.5,.45);
\draw[line width=\knotthickness] (-6.5, .45) to[out=90, in=90, looseness=.8] (.5,.5) to[out=-90, in=0, looseness=.8] (0,0);
\draw[line width=\knotthickness] (-3.5,.45) to[out=90, in=90, looseness=1.4] (-2.5,.45);
\draw[line width=\knotthickness] (-1.5,.45) to[out=90, in=90, looseness=1.4] (-.5,.45);
\draw[line width=\knotthickness] (-6.5,-.45) to[out=-90, in=-90, looseness=1.4] (-5.5,-.45);
\draw[line width=\knotthickness] (-4.5,-.45) to[out=-90, in=-90, looseness=1.4] (-3.5,-.45);
\draw[line width=\knotthickness] (-2.5,-.45) to[out=-90, in=-90, looseness=1.4] (-1.5,-.45);
\draw[line width=\knotthickness] (-7,0) to[out=180, in=90, looseness=.8] (-7.5,-.5) to[out=-90, in=-90, looseness=.8] (-.5,-.45)  ;
\end{tikzpicture}
\caption{The $7_1$ knot as a $1 \x 7$ virtual mosaic and its closure}
\label{fig:Vmosaic Closure}
\end{figure}
\end{center}
\end{example}

Ganzell and Henrich introduced a collection of \emph{virtual mosaic moves} on square mosaics that preserve knot type \cite{GH}. These moves are categorized as \emph{planar isotopies}, \emph{Reidemeister moves}, \emph{surface isotopies}, \emph{stabilization/destabilization}, and \emph{injection/ejection}. We refer the reader to \cite{GH} for these moves. Each of these moves apply in the rectangular case, and a straightforward adaptation of the injection and ejection moves gives a generalization of virtual mosaic moves to rectangular mosaics.

\subsection{Injection and Ejection}

Ganzell and Henrich's injection move and its inverse, the ejection move, will expand or contract the mosaic by adding or deleting two rows and/or columns. Below, we redefine the injection move using a relabeling system that allows us to give a well-defined injection function at any position along a square mosaic. We will then expand this definition to allow for rectangular injection and ejection. Let $\mathbb{V}^{(nm)}$ denote the set of $n \x m$ virtual rectangular mosaics. Let $V^{(nm)}$ denote a particular element of $\V^{(nm)}$ and denote the $ij^{th}$ entry of $V^{(nm)}$ by $V_{ij}^{(nm)}$.

\begin{definition}
A \emph{square virtual mosaic injection} is defined by the following function $\iota:~\V^{(nm)} \rightarrow \V^{(n+2,m+2)}$ given by:
\[
V_{ij}^{(n+2,m+2)}=
\begin{cases*}
V_{ij}^{(n)} & if $i \neq \alpha_1, \alpha_2, j \neq \beta_1, \beta_2$\\
T_5 & if $j = \beta_1, \beta_2$ and $V_{i,j^+-1}^{(n)}$ or $V_{i,j^+}^{(n)} \in\{T_2,T_3,T_5,T_7,T_8,T_9,T_{10}\}$ \\
T_6 & if $i = \alpha_1,\alpha_2$, and $V_{i^+-1,j}^{(n)}$ or  $V_{i^+,j}^{(n)}\in\{T_1,T_2,T_6,T_7,T_8,T_9,T_{10}\}$ \\
T_0 & $\vphantom{V_i^{(n)}}$otherwise.
\end{cases*}
\]
Here, the new rows created in the injection are called $\alpha_1, \alpha_2$ and the new columns called $\beta_1, \beta_2$. Further, $i^+$ and $j^+$ are the values of the location of the injection $\iota_{i^+j^+}$.
\end{definition}

\begin{example} Figure \ref{fig:inject} from \cite{GH} depicts an $\iota_{12}$ injection.  

\begin{figure}[H]
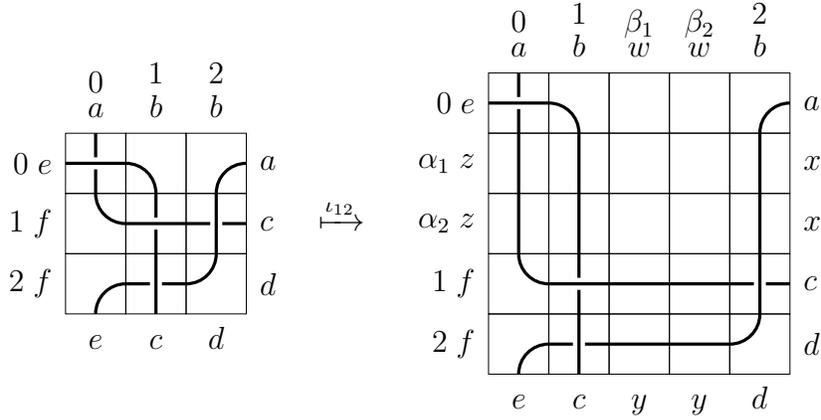

    \centering
\begin{vmosaic}[.8]{3}{3}{{\myoverset{0}{a},\myoverset{1}{b},\myoverset{2}{b}}}{{a,c,d}}{{d,c,e}}{{2\; f,1\;f,0\;e}} 
\tileix \& \tilei \& \tileii \\
\tileiii \& \tileix \& \tilex \\
\tileii \& \tilex \& \tileiv \\
\end{vmosaic}
\meq[1em]{\xmapsto{\iota_{12}}}
\begin{vmosaic}[.8]{5}{5}{{\myoverset{0}{a},\myoverset{1}{b},\myoverset{\beta_1}{w},\myoverset{\beta_2}{w},\myoverset{2}{b} }}{{a,x,x,c,d}}{{d,y,y,c,e}}{{2\;f,1\;f,\alpha_2\;z,\alpha_1\;z,0\;e}} 
\tileix \& \tilei \& \tileo \& \tileo \& \tileii \\
\tilevi \& \tilevi \& \tileo \& \tileo \& \tilevi \\
\tilevi \& \tilevi \& \tileo \& \tileo \& \tilevi \\
\tileiii \& \tileix \& \tilev \& \tilev \& \tilex \\
\tileii \& \tilex \& \tilev \& \tilev \& \tileiv \\
\end{vmosaic}
    \caption{An $\iota_{12}$-injection}
    \label{fig:inject}
\end{figure}

\end{example}

In the rectangular case, we can define \emph{row injections, row ejections, column injections,} and \emph{column ejections}. We define these virtual mosaic moves as follows.

\begin{definition}
A \emph{row injection}, $V_{i\x}$, injects two rows, $\alpha_1, \alpha_2$ after the $i^{th}$ row, $i=0,\dots, m$ and is defined as 

\[
V_{i\x}^{((m + 2) \times n)}=\begin{cases*}
V_{ij}^{(m \times n)} & if $i \neq \alpha_1, \alpha_2$\\
T_6 & if $i = \alpha_1, \alpha_2$ and $V_{i^+-1,j}$ or $ V_{i^+,j}\in\{T_1,T_2,T_6,T_7,T_8,T_9,T_{10}\}$ \\
T_0 & $\vphantom{V_i^{(n)}}$otherwise.
\end{cases*}
\]

\end{definition}

Figure \ref{rinject} gives an example of a $V_{1\x}$ row injection.

\begin{figure}[H]
    \centering
\begin{vmosaic}[.8]{3}{3}{{\myoverset{0}{a},\myoverset{1}{b},\myoverset{2}{b}}}{{a,c,d}}{{d,c,e}}{{2\; f,1\;f,0\;e}} 
\tileix \& \tilei \& \tileii \\
\tileiii \& \tileix \& \tilex \\
\tileii \& \tilex \& \tileiv \\
\end{vmosaic}
\meq[1em]{\xmapsto{V_{1x}}}
\begin{vmosaic}[.8]{5}{3}{{\myoverset{0}{a},\myoverset{1}{b},\myoverset{2}{b} }}{{a,x,x,c,d}}{{d,c,e}}{{2\;f,1\;f,\alpha_2\;z,\alpha_1\;z,0\;e}} 
\tileix \& \tilei \& \tileii \\
\tilevi \& \tilevi \& \tilevi \\
\tilevi \& \tilevi \& \tilevi \\
\tileiii \& \tileix \& \tilex \\
\tileii \& \tilex \& \tileiv \\
\end{vmosaic}
    \caption{A $V_{1\x}$-injection}
    \label{rinject}
\end{figure}

\begin{definition}
A \emph{column injection}, $V_{\x j}$, inserts two columns $\beta_1, \beta_2$ after the $j^{th}$ column, $j=0,\dots,n$ is defined as 

\[
V_{\x j}^{(m \times (n + 2))}=\begin{cases*}
V_{ij}^{(m \times n)} & if $j \neq \beta_1, \beta_2$\\
T_5 & if $j = \beta_1, \beta_2$ and $V_{i,j^+-1}$ or $ V_{i,j^+}\in\{T_1,T_2,T_6,T_7,T_8,T_9,T_{10}\}$ \\
T_0 & $\vphantom{V_i^{(n)}}$otherwise.
\end{cases*}
\]
\end{definition}

An example of a column injection is shown in Figure \ref{cinject}.

\begin{figure}[H]
    \centering
\begin{vmosaic}[.8]{3}{3}{{\myoverset{0}{a},\myoverset{1}{b},\myoverset{2}{b}}}{{a,c,d}}{{d,c,e}}{{2\; f,1\;f,0\;e}} 
\tileix \& \tilei \& \tileii \\
\tileiii \& \tileix \& \tilex \\
\tileii \& \tilex \& \tileiv \\
\end{vmosaic}
\meq[1em]{\xmapsto{V_{x2}}}
\begin{vmosaic}[.8]{3}{5}{{\myoverset{0}{a},\myoverset{1}{b},\myoverset{\beta_1}{w},\myoverset{\beta_2}{w},\myoverset{2}{b} }}{{a,c,d}}{{d,y,y,c,e}}{{2\;f,1\;f\;z,0\;e}} 
\tileix \& \tilei \& \tileo \& \tileo \& \tileii \\
\tileiii \& \tileix \& \tilev \& \tilev \& \tilex \\
\tileii \& \tilex \& \tilev \& \tilev \& \tileiv \\
\end{vmosaic}
    \caption{A $V_{\x 2}$-injection}
    \label{cinject}
\end{figure}

These injection moves and their inverses preserve virtual knot type. As row and column injections add two new vertices and two new edge pairs, they also preserve genus.

\section{Invariants of Rectangular Mosaics}\label{sec:3}


\begin{definition}
    Given a virtual knot or link $K$, define the \emph{tile number} of $K$, denoted $\tau(K)$ to be the minimal number of tiles needed to represent $K$ on a virtual rectangular mosaic. 
\end{definition}

Immediately, we see that for a virtual knot $K$, $\tau(K) \le m_v(K)^2$. In many cases this inequality is strict. For example, the figure eight knot requires a $3 \x 3$ square mosaic and has $m_v(4_1) = 3.$ However, Figure \ref{fig8rec} shows a presentation of the $4_1$ knot as a $2 \x 3$ virtual mosaic, showing that $\tau(4_1) \le 8$.  

\begin{figure}[H]
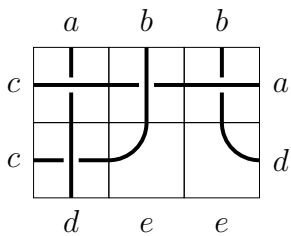

\centering
\begin{vmosaic}[1]{2}{3}{{a,b,b}}{{a,d}}{{e,e,d}}{{c,c}} 
\tileix \& \tilex \& \tileix \\
\tilex \& \tileiv \& \tileiii\\
\end{vmosaic}
\caption{The $4_1$ knot on a $2 \x 3$ virtual mosaic.}
\label{fig8rec}
\end{figure}

We provide a table of virtual rectangular mosaics for all classical knots up to 8 crossings at the end of this manuscript. 

 \subsection{Virtual Row Mosaics}

The classical crossing number $c(K)$ of a virtual knot $K$ provides a lower bound for the tile number. Further, knots with prime classical crossing number $p$ can only realize this minimum if they are able to be presented on a $1 \x p$ rectangular mosaic. 

\begin{definition}
    A \emph{row mosaic} for a virtual knot or link $K$ is a $1 \x n$ virtual rectangular mosaic representing $K$. 
\end{definition}

\begin{proposition}
    The figure eight knot has tile number $\tau(4_1) = 4$, as it can be presented as a $1 \x 4$ row mosaic, as seen in Figure \ref{fig:4_1row}.
\begin{figure}[h]
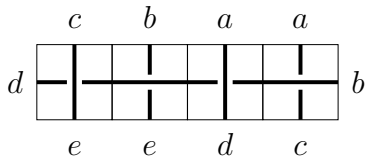

\centering
\begin{vmosaic}[1]{1}{4}{{c,b,a,a}}{{b}}{{c,d,e,e}}{{d}} 
\tilex \& \tileix \& \tilex \& \tileix \\
\end{vmosaic} 
\caption{A row mosaic representation of the $4_1$ knot.}
   \label{fig:4_1row} 
   \end{figure}
\end{proposition}

\begin{example}
    In Figure \ref{rowex}, we see an example of a row mosaic representing the classical knot $7_1$. Thus, $\tau(7_1) = 7$. 
\end{example}

\begin{figure}[h]
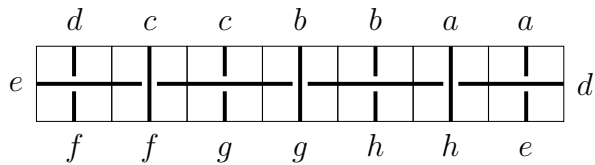

\centering
\begin{vmosaic}
[1]{1}{7}{{d,c,c,b,b,a,a}}{{d}}{{e,h,h,g,g,f,f}}{{e}} 
\tileix \& \tilex \& \tileix \& \tilex \& \tileix \& \tilex \& \tileix \\
\end{vmosaic}
\caption{A row mosaic representation of the $7_1$ knot.}
\label{rowex}
\end{figure}

Representing a virtual knot as a row mosaic requires aligning the classical crossings in a row. A natural question arises as to if every virtual knot $K$ can be represented on a row mosaic. We answer this question in the affirmative utilizing the Gauss code. Recall that for a virtual knot $K$, a Gauss code $\omega_K$ is created by labeling the classical crossings of a virtual knot diagram $1, \dots, n$ and then traversing the diagram, recording the crossing information as a sequence. At each crossing, we indicate if we are traversing the overstrand or the understrand with an $O$ or $U$, followed by the number of the crossing, followed by the sign $+$ or $-$ of the crossing. If the knot is classical, we can omit the over- and under- information, as there will be a unique realization of the Gauss code that omits virtual crossings. 

\begin{figure}[h]
\centering
\includegraphics[scale=.4]{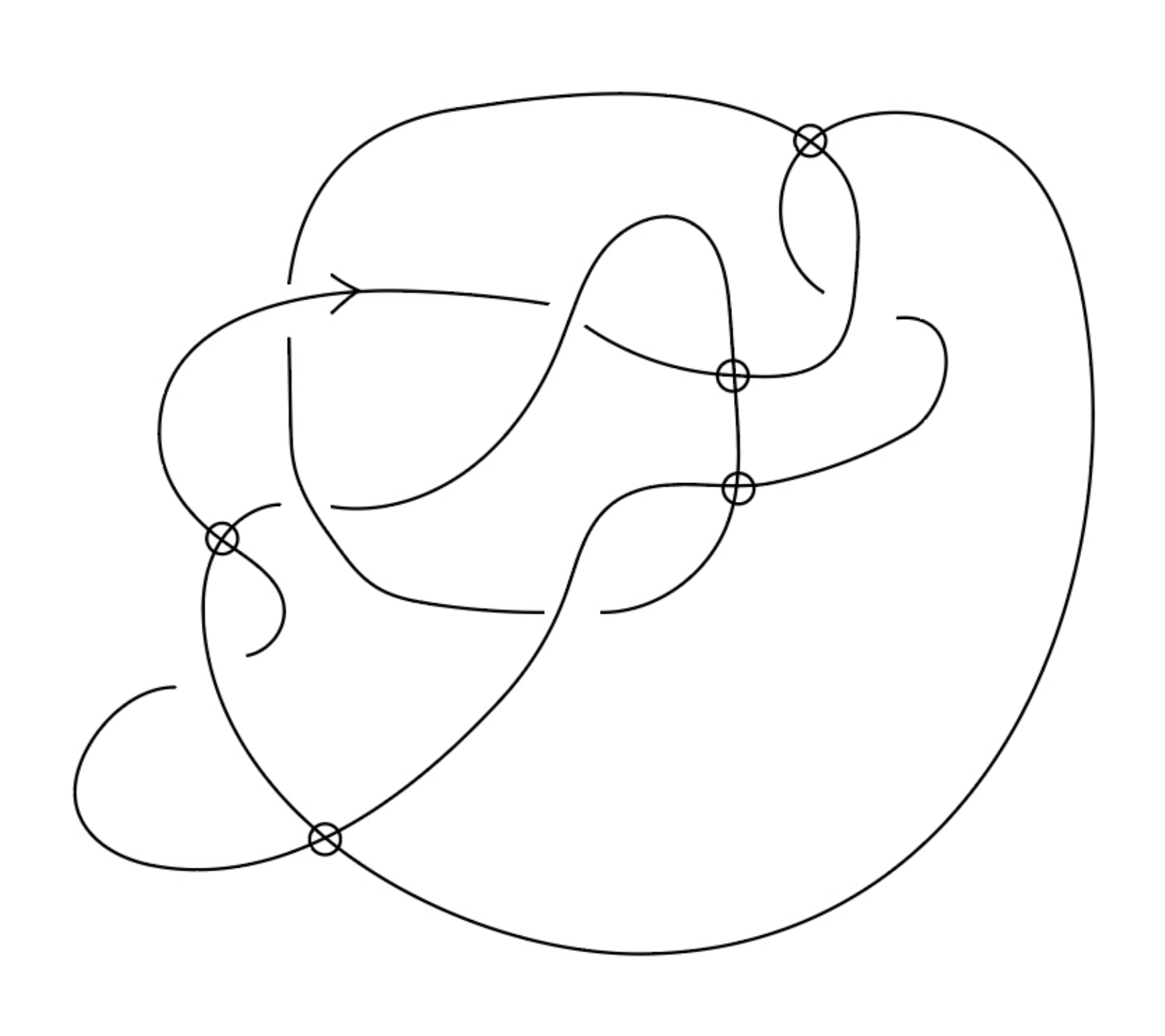}
\caption{An example of a virtual knot $K$ with 6 classical crossings where $M(\omega_K)=5$.}
\label{VK}
\end{figure}

\begin{example}
    The virtual knot $K$ pictured in Figure \ref{VK} has Gauss code given by $$\omega_K=O1-{U2+O3+U1-O4-U5+}O2+U4-O6+U3+O5+U6+.$$
\end{example}

\begin{definition} 
    Given a Gauss code $\omega_K$ for a knot $K$ with $n$ classical crossings, we define $M(\omega_K)$ to be the length of the maximal sequence of nonrepeating terms of $\omega_K$. Thus, for all $\omega$, $M(\omega_K) \le n$. Note that $M(\omega_K)$ is not an invariant of virtual knots.
\end{definition}

\begin{example}
Consider the  virtual knot $K$ in Figure \ref{VK} realizing the Guass code $\omega_K=O1-\textbf{U2+O3+U1-O4-U5+}O2+U4-O6+U3+O5+U6+$. We see that $K$ has 6 classical crossings, while $M(\omega) = 5$. 
\end{example}

\begin{theorem}\label{thm:VirtualRows}
Given a virtual knot $K$, there exists an $n$ such that $K$ can be represented on a $1 \x n$ virtual mosaic.
\end{theorem}

\begin{proof}
    We produce an algorithm for representing a virtual knot $K$ on a row mosaic using a Gauss code for $K$ as a guide. First, suppose that $K$ is a virtual knot with $n$ classical crossings and Gauss code given by $\omega_K = \alpha_1 \pm \alpha_2 \pm \dots \alpha_{2n}$. Further assume that $\omega_K$ begins with a maximal sequence of nonrepeated crossings of length $m \le n$. Then, a row mosaic for $K$ begins with $m$ crossings tiles -  $T_9$ or $T_{10}$ - according to the sign of $\alpha_i$. Following the Gauss code, $\alpha_{m+1}$ is a repeated crossing. Insert a $T_8$ tile as the $m+1^{st}$ tile and label the upper edge of the $T_8$ tile and the connecting arc of the repeated tile as $a$. If $\alpha_{m+2}$ is a repeated crossing, label the edges of the crossing tiles to reflect the connection via an edge identification. If $\alpha_{m+2}$ is not a repeated crossing, insert a $T_9$ or $T_{10}$ tile to realize the new crossing and label the corresponding edges accordingly. Repeat this process until all crossings have been included and all edge identifications realize the given Gauss code. 
\end{proof}

\begin{example}
The algorithm from Theorem \ref{thm:VirtualRows} produces the $1\x 7$ row mosaic shown in Figure \ref{fig: row alg} representing the knot $K$ pictured in Figure \ref{VK}.

\begin{figure}
\hspace{-4in}
\begin{tikzpicture}
\begin{vmosaic}[1]{1}{7}{{a,e,h,b,e,c,g}}{{d}}{{f,a,f,c,g,d,b}}{{h}} 
    \tilepx \& \tilepix \& \tilenx \& \tilenix \& \tilepx \& \tileviio \& \tilepix \\
    \end{vmosaic}
    \node at (-6.8,-.2) {2};
    \node at (-5.8, -.2) {3};
    \node at (-4.8, -.2) {1};
    \node at (-3.8, -.2) {4};
    \node at (-2.8, -.2) {5};
    \node at (-.8, -.2) {6};
\end{tikzpicture}
\caption{A row mosaic for $K$}
\label{fig: row alg}
\end{figure}
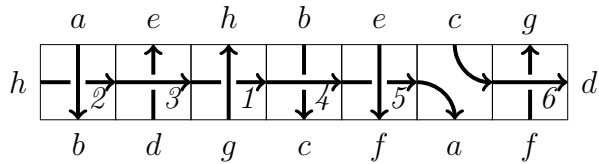
\end{example}

Using this algorithm, we produce a table of virtual row mosaics for all virtual knots with up to 4 classical crossings, included as an appendix.

\begin{definition}
    For a virtual knot $K$, we define the \emph{row number} of $K$, denoted $\rho(K)$, to be the minimal $n$ such that $K$ can be represented on a $1 \x n$ row mosaic.
\end{definition}

We then have the following corollary. 

\begin{corollary}
If a virtual knot $K$ with classical crossing number $c(K)=n$ has a Gauss code $\omega_K$ with maximal sequence $M(\omega_K)=n$, then $n=\tau(K)=\rho(K)$.
\end{corollary}

We can consider classical knots to be the virtual knots that can be represented on a genus 0 thickened surface. However, we note that the algorithm provided in Theorem \ref{thm:VirtualRows}, when applied to a classical knot $K$ does not necessarily produce a classical projection; virtual crossings may arise. 

\begin{example}
Consider the $8_{16}$ knot with Gauss code $$\omega(8_{16})= 1-8+5-6+{\bf 2-1+4-5+6 -7+3} -4+ 8 -2+ 7-3 .$$ Theorem \ref{thm:VirtualRows} gives a row mosaic as pictured in Figure \ref{fig:8-16}, which contains virtual crossings. 
\begin{figure}
    \centering 

\vspace{.5in}
\hspace{-5.5 in}
\begin{tikzpicture}
\begin{vmosaic}[1]{1}{9}{{a,d,i,f,h,c,c,j,f}}{{a}}{{e,i,d,b,g,g,j,e,b}}{{h}} 
    \tileix \& \tilex \& \tileix \& \tilex \& \tileix \& \tilex \& \tileix \& \tileviii \& \tileix \\
    \end{vmosaic}
    \node at (-8.8, -.2) {2};
    \node at (-7.8, -.2) {1};
    \node at (-6.8,-.2) {4};
    \node at (-5.8, -.2) {5};
    \node at (-4.8, -.2) {6};
    \node at (-3.8, -.2) {7};
    \node at (-2.8, -.2) {3};
    \node at (-.8, -.2) {8};
\end{tikzpicture}
\caption{The $8_{16}$ knot as a row mosaic.}
\label{fig:8-16}
\end{figure}
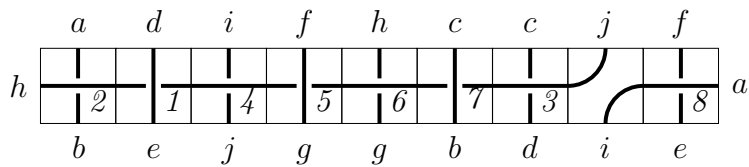
\end{example}

Building on work of Owad \cite{Owad} and Adams, Shinjo, and Tanaka \cite{Adams}, we show that every classical knot can be represented as a genus zero row mosaic. 

\begin{theorem}
    Any classical knot $K$ can be represented on a genus zero virtual row mosaic.
\end{theorem}

\begin{proof}
    The existence of a genus zero row mosaic for a classical knot requires aligning the crossings of a classical knot projection.  Work of Owad \cite{Owad} shows that every classical knot can be presented as a \emph{straight knot} which we then place onto a virtual row mosaic using only $T_9$ and $T_{10}$ crossing tiles. The algorithm for doing so relies on the following result, stated as Theorem 1.2 of \cite{Adams}:

    \begin{theorem} Every knot has a projection that can be decomposed into two sub-arcs such that each sub-arc never crosses itself. 
    \end{theorem}

The existence of such a projection dates back to 1960 \cite{Hotz} and was rediscovered in 2007 \cite{Ozawa} before motivating the above Theorem. Owad \cite{Owad} then uses planar isotopy to can ensure that one of the sub-arcs is linear; this configuration aligns all crossings in a row. Such as projection is called a \emph{straight position} for a knot. The non-linear sub-arc then consists of a collections of non-intersecting semi-circles. An example is provided in Figure \ref{arctostraight}.

\begin{figure}[H]
    \centering
     \includegraphics[width=0.8\textwidth]{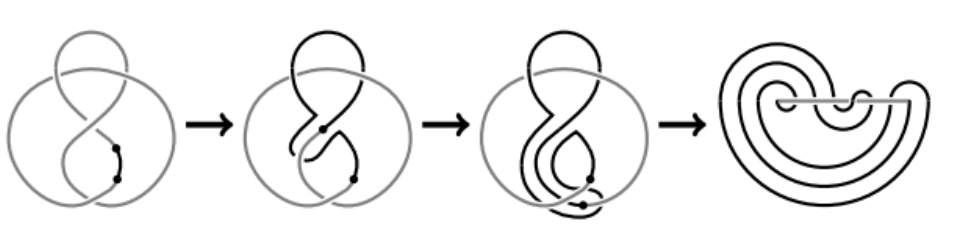} 
    \caption{Isotoping the Figure 8 knot into straight position. Figure from \cite{Owad}}
    \label{arctostraight}
\end{figure}

A knot $K$ in straight position will have a Gauss code $\omega(K)$ in which the maximal sequence of nonrepeated crossings $M(\omega(K))$ is the crossing number of the knot diagram. Thus if a knot is in straight position, our algorithm produces a $1 \x n$ row mosaic for $K$ using only $T_9$ and $T_{10}$ crossing tiles. The boundary edges of the mosaic are then labeled such that each semi-circle of the straight diagram pairs two boundary edges with opposite orientation. 


Recall that surface will only have positive genus if it is the quotient of a polygonal presentation with an interlocking set of edge pairs, i.e. $\alpha_i \cdots \alpha_j \cdots \alpha_i \cdots \alpha_j$. As every edge in a row mosaic for a knot in straight position contains an arc of the knot diagram, a set of interlocking edge pairs would produce a virtual crossing in the row mosaic, as shown in Figure \ref{crosspair}. Thus, we know the edge labeling cannot include an interlocking pair and so the row mosaic has genus zero.

\end{proof}

Putting a knot into straight position may not realize the crossing number. Owad \cite{Owad} gives the following definition. 

\begin{definition}
    A (classical) knot $K$ with crossing number $n$ is called \emph{perfectly straight} if there exists an $n$-crossing projection for $K$ in straight position. 
\end{definition}

This gives the following consequence.
\begin{corollary}
    For a classical knot $K$ with crossing number $n$, if $K$ is perfectly straight, then $c(K)=\tau(K)=\rho(K)=n$.
\end{corollary}

Several families of knots are perfectly straight, including torus knots $T_{p,q}$, $n$-pretzel knots, classical knots with crossing number up to 7, and 2-Bridge knots $K_{p/q}$ where the continued fraction of $p/q$ has length less than 6 \cite{Owad}. However, several perfectly straight knots can realize their tile number in both a row mosaic and a rectangular mosaic with more than one row. The $6_3$ knot is an example of such; both representations are shown in Figure \ref{roweqtile}. We include as an appendix a table of all alternating classical knots that can be realized on a $2 \x 3$ mosaic and all knots that can be realized on a $2 \x 4$ mosaic; these were constructed combinatorially using Python. Knots $3_1, 4_1, 5_2$, and $6_3$ are the only classical knots that can be realized as a $2 \x 3$ mosaic, while $5_1, 6_1, 6_2, 7_4, 7_5, 7_6, 7_7,$ and $8_{15}$ are the alternating classical knots that can be realized as a $2 \x 4$ mosaic. Thus, the $8_{15}$ knot is the only other alternating knot with fewer than 9 crossings whose tile number can be realized on mosaics of two different dimensions.

\begin{figure}[H]
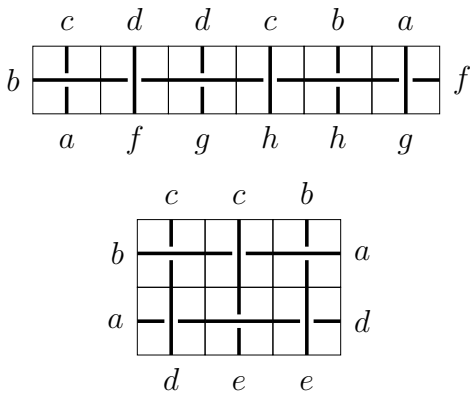

\centering
\begin{vmosaic}[.9]{1}{6}{{c,d,d,c,b,a}}{{f}}{{g,h,h,g,f,a}}{{b}} 
\tileix \& \tilex \& \tileix \& \tilex \& \tileix \& \tilex \\
\end{vmosaic}\\
\begin{vmosaic}[.9]{2}{3}{{c,c,b}}{{a,d}}{{e,e,d}}{{a,b}} 
\tileix \& \tilex \& \tileix\\
\tilex \& \tileix \& \tilex \\
\end{vmosaic}\\
\caption{The $6_3$ knot represented on a row mosaic and rectangular mosaic.}
\label{roweqtile}
\end{figure}

A natural question to ask is if the tile number of a virtual knot will ever be strictly less than the row number. The first such example is the $10_{88}$ knot, whose row  number is 12, as determined by \cite{Owad}, but can be realized on a $2 \x 5$ rectangular mosaic as shown in Figure \ref{rowneqtile}.

\begin{figure}[H]
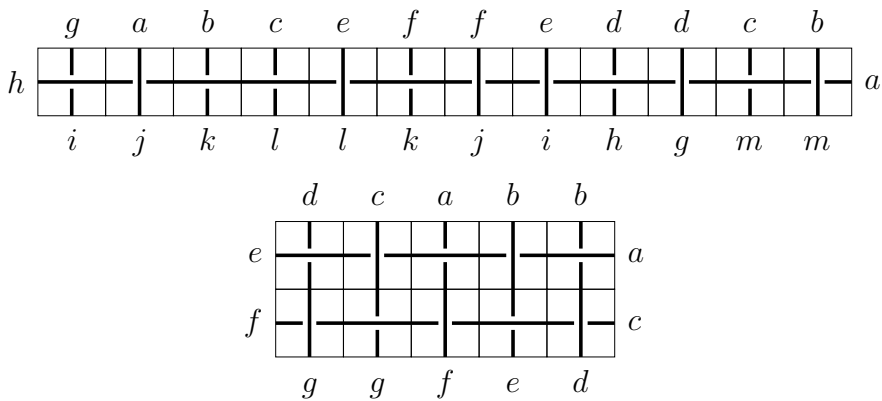

\centering
\begin{vmosaic}[.9]{1}{12}{{g,a,b,c,e,f,f,e,d,d,c,b}}{{a}}{{m,m,g,h,i,j,k,l,l,k,j,i}}{{h}} 
\tileix \& \tilex \& \tileix \& \tileix \& \tilex \& \tileix \& \tilex \& \tilex \& \tileix \& \tilex \& \tileix \& \tilex \\
\end{vmosaic}

\begin{vmosaic}[.9]{2}{5}{{d,c,a,b,b}}{{a,c}}{{d,e,f,g,g}}{{f,e}}
\tileix \& \tilex \& \tileix \& \tilex \& \tileix\\
\tilex \& \tileix \& \tilex \& \tileix \& \tilex\\
\end{vmosaic}
\caption{The $10_{88}$ knot on a $1 \x 12$ and $2 \times 5$ mosaic.}
\label{rowneqtile}
\end{figure}


\section{Virtual Polynomial Invariants}\label{sec:4}

Presenting virtual knots as row mosaics yields several computational benefits, particularly in the world of polynomial computations. Several virtual knot invariants are computed based on algorithms involving assigning \emph{weights} to classical crossings of virtual knot diagrams and constructing polynomial invariants \cite{Cheng, Henrich, Kauffman, Higa}. These include the \emph{Affine Index Polynomial} \cite{Kauffman}, the \emph{writhe polymomial} \cite{ChengGao}, a family of \emph{Intersection Polynomials} \cite{Higa}, the $F$-\emph{polynomial} \cite{Kaur}, the \emph{Zero-polynomial}\cite{zero}, and more. We discuss the construction of the \emph{intersection index polynomial} as presented in \cite{dover} as a representative of these similarly computed families of intersection-based polynomial invariants. 

Given an oriented virtual knot $K$ with a classical crossing $d$, the act of smoothing $d$ produces a 2-component link $K_d$ with components labeled 1 or 2 with a convention of designating component 1 to be on the left when the smoothed crossing is oriented upward as in Figure \ref{smooth}.

\begin{figure}[H]
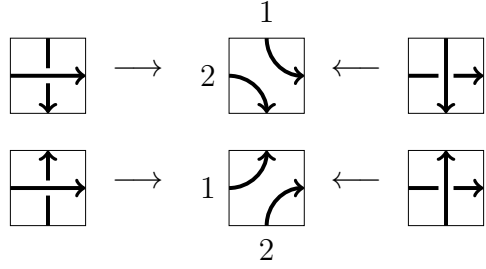

    \centering
    \begin{tabular}{ccccc}
      \begin{vmosaic}[1]{1}{1}{{}}{{}}{{}}{{}} \tilenix\\ 
      \end{vmosaic} &
      $\longrightarrow$ & 
      \begin{vmosaic}[1]{1}{1}{{1}}{{}}{{}}{{2}} \tileviio\\ 
      \end{vmosaic}  & 
      $\longleftarrow$ &
      \begin{vmosaic}[1]{1}{1}{{}}{{}}{{}}{{}} \tilepx\\ 
      \end{vmosaic} \\
       & &  & & \\
      \begin{vmosaic}[1]{1}{1}{{}}{{}}{{}}{{}} \tilepix\\ 
      \end{vmosaic} & 
      $\longrightarrow$ & 
      \begin{vmosaic}[1]{1}{1}{{}}{{}}{{2}}{{1}} \tileviiio\\ 
      \end{vmosaic}  & 
      $\longleftarrow$ &
      \begin{vmosaic}[1]{1}{1}{{}}{{}}{{}}{{}} \tilenx\\ 
      \end{vmosaic} \\
       & &  & & \\
    \end{tabular}
    \caption{Smoothing of a crossing.}
    \label{smooth}
\end{figure}

 Let $C_d$ be the set of classical crossings in the virtual link diagram that involve both components after crossing $d$ is smoothed. Then each crossing $x \in C_d$ is assigned a value $\alpha(x)$, where $\alpha(x)= 1$ if component 1 of $K_d$ passes from left to right when orientations point upward and $\alpha(x) = -1$ if component 2 passes from left to right when orientations point upward. The \emph{the intersection index} corresponding to crossing $d$, denoted $i(d)$, is the sum of the values $\alpha(x)$ for all classical crossings $x \in C_d$. Thus, $i(d) =\displaystyle \sum_{x \in C_d} \alpha(x)$. 

\begin{definition}
The \emph{intersection index polynomial}, $p_t(K)$, as defined in \cite{dover} for a virtual knot $K$ with diagram $D$ is a sum over all classical crossings $d \in D$, where $s(d)$ is the sign of crossing $d$, given as follows:
\begin{center}
   $\displaystyle p_t(K)=\sum_{d \in D} s(d)(t^{|i(d)|}-1)$.
\end{center}
\end{definition}

By orienting a virtual row mosaic in such a way that the horizontal strand is oriented from left to right, and algorithmically smoothing each crossing via a tile replacement as depicted in Figure \ref{soct}, we can compute the intersection index polynomial for the virtual knot. An example is provided below.  

\begin{figure}[H]
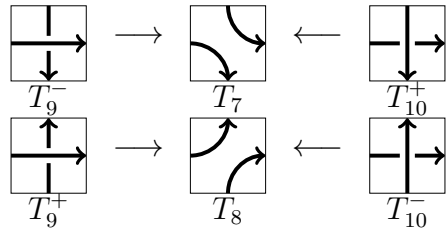

    \centering
    \begin{tabular}{ccccc}
      \begin{vmosaic}[1]{1}{1}{{}}{{}}{{}}{{}} \tilenix\\ 
      \end{vmosaic} &
      $\longrightarrow$ & 
      \begin{vmosaic}[1]{1}{1}{{}}{{}}{{}}{{}} \tileviio\\ 
      \end{vmosaic}  & 
      $\longleftarrow$ &
      \begin{vmosaic}[1]{1}{1}{{}}{{}}{{}}{{}} \tilepx\\ 
      \end{vmosaic} \\
      $T_9^-$ & & $T_7$ & & $T_{10}^+$\\
      \begin{vmosaic}[1]{1}{1}{{}}{{}}{{}}{{}} \tilepix\\ 
      \end{vmosaic} & 
      $\longrightarrow$ & 
      \begin{vmosaic}[1]{1}{1}{{}}{{}}{{}}{{}} \tileviiio\\ 
      \end{vmosaic}  & 
      $\longleftarrow$ &
      \begin{vmosaic}[1]{1}{1}{{}}{{}}{{}}{{}} \tilenx\\ 
      \end{vmosaic} \\
      $T_9^+$  & & $T_8$ & & $T_{10}^-$\\
         \end{tabular}
    \caption{Smoothing oriented crossing tiles.}
    \label{soct}
\end{figure}

 \begin{example}
 Figure \ref{iimosaic} shows a virtual knot $K$ with four classical crossings on a row mosaic as well as the four smoothed diagrams. The resulting intersection index polynomial is given by 
  $$p_t(K) = -t^3-3t+4.$$ 
  
\vspace{-.3in}
\begin{figure}[H]
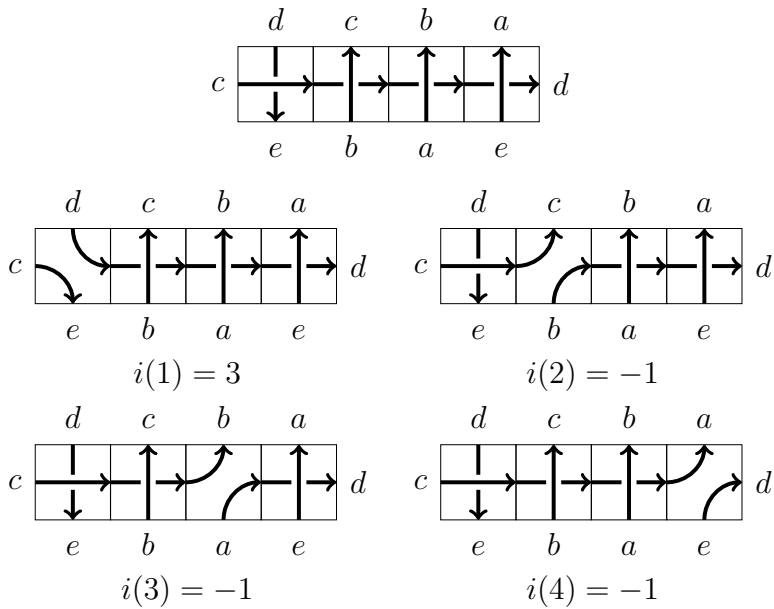

\centering
\begin{vmosaic}[1]{1}{4}{{d,c,b,a}}{{d}}{{e,a,b,e}}{{c}} 
\tilenix \& \tilenx \& \tilenx \& \tilenx\\
\end{vmosaic}

\begin{tabular}{cc}
\begin{vmosaic}[1]{1}{4}{{d,c,b,a}}{{d}}{{e,a,b,e}}{{c}} 
\tileviio \& \tilenx \& \tilenx \& \tilenx\\
\end{vmosaic} &
\begin{vmosaic}[1]{1}{4}{{d,c,b,a}}{{d}}{{e,a,b,e}}{{c}}
\tilenix \& \tileviiio \& \tilenx \& \tilenx\\
\end{vmosaic}\\
$i(1) = 3$ & $i(2) = -1$\\
\begin{vmosaic}[1]{1}{4}{{d,c,b,a}}{{d}}{{e,a,b,e}}{{c}} 
\tilenix \& \tilenx \& \tileviiio \& \tilenx\\
\end{vmosaic} &
\begin{vmosaic}[1]{1}{4}{{d,c,b,a}}{{d}}{{e,a,b,e}}{{c}}
\tilenix \& \tilenx \& \tilenx \& \tileviiio\\
\end{vmosaic}\\
$i(3) = -1$ & $i(4) = -1$\\
\end{tabular}

\caption{Example of computing the intersection index polynomial on a virtual mosaic.}
\label{iimosaic}
\end{figure}
\end{example}

As shown in Figure \ref{cc}, we smooth the first crossing and label component 1 in blue and component 2 in red. The intersection index of the first crossing, $i(1)$, is computed using the crossings that involve both components. The blue component passes from left to right and all crossings are oriented up. Thus, $i(1) = 3$.  

\begin{figure}[H]
\hspace{-3in}
\centering
\begin{tikzpicture}
\begin{vmosaic}[1]{1}{4}{{}}{{}}{{}}{{}}
\tileviioa \& \tilenxa \& \tilenxa \& \tilenxa\\
\end{vmosaic}
\draw[blue,line width=\knotthickness] (-3.5, .45) to[out=90, in=90, looseness=.8] (.5,.5) to[out=-90, in=0, looseness=.8] (0,0);
\draw[red,line width=\knotthickness] (-.5, .5) to[out=90, in=90, looseness=1.5] (.8,.5) to[out=-90, in=-95, looseness=1.5] (-1.5,-.5);
\draw[red,line width=\knotthickness] (-1.5, .5) to[out=90, in=90, looseness=2] (1.2,.5) to[out=-90, in=-90, looseness=2] (-2.5,-.5);
\draw[red,line width=\knotthickness] (-3.5,-.45) to[out=-90, in=-90, looseness=.8] (-.5,-.45);
\draw[red,line width=\knotthickness] (-4,0) to[out=-180, in=-90, looseness=.8] (-4.5,.5) to[out=90, in=90, looseness=.8] (-2.5,.45)  ;
\end{tikzpicture}
\caption{A color coded smoothing of the first crossing.}
\label{cc}
\end{figure}
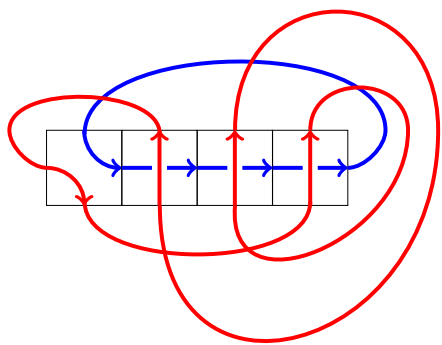

By systematically smoothing the classical crossings from left to right in a virtual row mosaic, we can hope to make these complex computations more tractable. The second author has written Python code to compute the intersection index polynomial for row mosaics with only crossing tiles, available at:\\ \texttt{https://github.com/rmeye23/Intersection-Index-Polynomial-Row-Mosaics.}

There exist nontrivial virtual knots whose crossings all have intersection index zero; these knots have trivial intersection index polynomial. Further, a crossing with intersection index zero retains its index when switching the under and over strands. This fact is of interest to those studying the Cosmetic Crossing conjecture, which hypothesizes that there are no non-nugatory crossing changes that preserve knot type. This conjecture has been proven true for several families of knots, but those with trivial intersection index polynomial remain unknown. Properties of virtual row mosaics may be useful in finding families of virtual knots with trivial intersection index polynomial. Figure \ref{zeroex} shows two nontrivial virtual knots with similar patterns; all crossings have intersection index zero. This pattern can be iterated to produce a family of nontrivial virtual knots with vanishing intersection index polynomial. 

\begin{figure}[H]
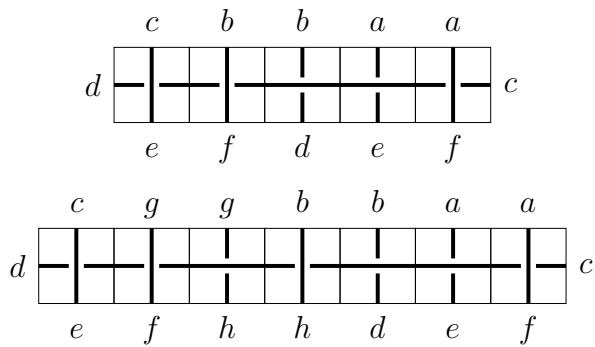

\centering
\begin{vmosaic}[1]{1}{5}{{c,b,b,a,a}}{{c}}{{f,e,d,f,e}}{{d}} 
\tilex \& \tilex \& \tileix \& \tileix \& \tilex \\
\end{vmosaic}

\begin{vmosaic}[1]{1}{7}{{c,g,g,b,b,a,a}}{{c}}{{f,e,d,h,h,f,e}}{{d}} 
\tilex \& \tilex \& \tileix \& \tilex \& \tileix \& \tileix \& \tilex\\
\end{vmosaic}

\caption{Virtual knots whose crossings are all weight zero and are not the unknot.}
\label{zeroex}
\end{figure}

 \section{Open Questions}\label{sec:5}
Opening the study of virtual mosaic knot theory to the rectangular setting introduces a number of interesting questions for further research. We propose a several here:
\begin{enumerate}
    \item Does every virtual knot have a rectangular mosaic that realizes both tile number and genus? We know this is true for virtual knots with four crossings or less. 
    \item Knots with prime classical crossing number necessarily have tile number equal to their row number. We have presented the $10_{88}$ knot as an example of a classical knot with row number greater than its tile number. What characteristics are necessary to guarantee that a knot $K$ will have $\rho(K) > \tau(K)$? Find a family of such knots.
    \item Based on our algorithm in Theorem \ref{thm:VirtualRows}, we can use a Gauss code for a virtual knot $K$ to determine an upper bound for the row number, $\rho(K)$. We can ask how efficient this algorithm is. If $K$ is a virtual knot with classical crossing number $n$, does there exist a constant $a_n$ such that $n \le \rho(K) \le n+a_n$?
    \item Similarly, can we find a bound for the tile number of a virtual knot based on its classical crossing number? Given a virtual knot $K$ with classical crossing number $n$, does there exist a constant $b_n$ such that $n \le \tau(K) \le n+b_n$?
    \item Identify other families of virtual knots with vanishing intersection index polynomials. Can the index of a crossing be easily computed from the row mosaic? 
\end{enumerate}

 \section{Table of Row Mosaics of Knots $3_1$ to $8_{21}$}

\label{minimal-mosaics}
\noindent
\resizebox{\textwidth}{!}{%
\begin{tabular}{ccc}
\begin{vmosaic}[1]{1}{3}{{b,a,a}}{{b}}{{c,d,d}}{{c}} 
\tileix \& \tilex \& \tileix \\
\end{vmosaic} &  
\begin{vmosaic}[1]{1}{4}{{c,b,a,a}}{{b}}{{c,d,e,e}}{{d}} 
\tilex \& \tileix \& \tilex \& \tileix \\
\end{vmosaic} &
\begin{vmosaic}[1]{1}{5}{{b,c,c,a,a}}{{b}}{{d,f,f,e,e}}{{d}} 
\tilex \& \tileix \& \tilex \& \tileix \& \tilex \\
\end{vmosaic} \\
$3_1$ & $4_1$ & $5_1$ 
\end{tabular}
}
\vspace{.2 in}

\medskip\noindent
\resizebox{\textwidth}{!}{%
\begin{tabular}{cc}
\begin{vmosaic}[1]{1}{5}{{c,a,b,b,a}}{{c}}{{f,f,d,e,e}}{{d}} 
\tilex \& \tileix \& \tilex \& \tileix \& \tilex \\
\end{vmosaic} &
\begin{vmosaic}[1]{1}{6}{{d,c,b,b,a,a}}{{c}}{{d,g,g,e,f,f}}{{e}} 
\tilex \& \tileix \& \tilex \& \tileix \& \tilex \& \tileix \\
\end{vmosaic}\\
$5_2$ & $6_1$
\end{tabular}
}
\vspace{.2 in}

\medskip\noindent
\resizebox{\textwidth}{!}{%
\begin{tabular}{cc}
\begin{vmosaic}[1]{1}{6}{{d,c,a,b,b,a}}{{c}}{{d,e,g,g,f,f,}}{{e}} 
\tileix \& \tilex \& \tileix \& \tilex \& \tileix \& \tilex \\
\end{vmosaic} &
\begin{vmosaic}[1]{1}{6}{{c,d,d,c,b,a}}{{f}}{{g,h,h,g,f,a}}{{b}} 
\tileix \& \tilex \& \tileix \& \tilex \& \tileix \& \tilex \\
\end{vmosaic}\\
$6_2$ & $6_3$
\end{tabular}
}
\vspace{.2 in}

\medskip\noindent
\resizebox{\textwidth}{!}{%
\begin{tabular}{cc}
\begin{vmosaic}[1]{1}{7}{{d,c,c,b,b,a,a}}{{d}}{{e,h,h,g,g,f,f}}{{e}} 
\tileix \& \tilex \& \tileix \& \tilex \& \tileix \& \tilex \& \tileix \\
\end{vmosaic} &
\begin{vmosaic}[1]{1}{7}{{d,a,c,c,b,b,a}}{{d}}{{h,h,g,g,e,f,f}}{{e}} 
\tilex \& \tileix \& \tilex \& \tileix \& \tilex \& \tileix \& \tilex \\
\end{vmosaic}\\
$7_1$ & $7_2$
\end{tabular}
}
\vspace{.2 in}

\medskip\noindent
\resizebox{\textwidth}{!}{%
\begin{tabular}{cc}
\begin{vmosaic}[1]{1}{7}{{d,d,c,b,b,a,a}}{{c}}{{e,h,h,f,g,g,f}}{{e}} 
\tilex \& \tileix \& \tilex \& \tileix \& \tilex \& \tileix \& \tilex \\
\end{vmosaic} &
\begin{vmosaic}[1]{1}{7}{{d,a,c,c,b,b,a}}{{d}}{{h,h,e,f,g,g,f}}{{e}} 
\tilex \& \tileix \& \tilex \& \tileix \& \tilex \& \tileix \& \tilex \\
\end{vmosaic}\\
$7_3$ & $7_4$
\end{tabular}
}
\vspace{.2 in}

\medskip\noindent
\resizebox{\textwidth}{!}{%
\begin{tabular}{cc}
\begin{vmosaic}[1]{1}{7}{{d,a,b,c,c,b,a}}{{d}}{{h,h,e,g,g,f,f}}{{e}} 
\tileix \& \tilex \& \tileix \& \tilex \& \tileix \& \tilex \& \tileix \\
\end{vmosaic} &
\begin{vmosaic}[1]{1}{7}{{e,e,d,c,b,a,a}}{{b}}{{c,f,g,h,h,g,f}}{{d}} 
\tileix \& \tilex \& \tileix \& \tilex \& \tileix \& \tilex \& \tileix \\
\end{vmosaic}\\
$7_5$ & $7_6$
\end{tabular}
}
\vspace{.2 in}

\medskip\noindent
\resizebox{\textwidth}{!}{%
\begin{tabular}{cc}
\begin{vmosaic}[1]{1}{7}{{e,d,c,a,b,b,a}}{{c}}{{d,e,f,g,h,h,g}}{{f}} 
\tilex \& \tileix \& \tilex \& \tileix \& \tilex \& \tileix \& \tilex \\
\end{vmosaic} &
\begin{vmosaic}[1]{1}{8}{{e,d,c,c,b,b,a,a}}{{d}}{{e,i,i,h,h,f,g,g}}{{f}} 
\tilex \& \tileix \& \tilex \& \tileix \& \tilex \& \tileix \& \tilex \& \tileix \\
\end{vmosaic}\\
$7_7$ & $8_1$
\end{tabular}
}
\vspace{.2 in}

\medskip\noindent
\resizebox{\textwidth}{!}{%
\begin{tabular}{cc}
\begin{vmosaic}[1]{1}{8}{{e,e,d,d,c,b,a,a}}{{b}}{{c,f,g,i,i,h,h,g}}{{f}} 
\tilex \& \tileix \& \tilex\& \tileix \& \tilex \& \tileix \& \tilex \& \tileix \\
\end{vmosaic} &
\begin{vmosaic}[1]{1}{8}{{e,d,d,c,b,b,a,a}}{{c}}{{e,i,i,f,h,h,g,g}}{{f}} 
\tilex \& \tileix \& \tilex \& \tileix \& \tilex \& \tileix \& \tilex \& \tileix \\
\end{vmosaic}\\
$8_2$ & $8_3$
\end{tabular}
}
\vspace{.2 in}

\medskip\noindent
\resizebox{\textwidth}{!}{%
\begin{tabular}{cc}
\begin{vmosaic}[1]{1}{8}{{e,e,d,c,b,b,a,a}}{{c}}{{d,i,i,f,g,h,h,g}}{{f}} 
\tileix \& \tilex \& \tileix\& \tilex \& \tileix \& \tilex \& \tileix \& \tilex \\
\end{vmosaic} &
\begin{vmosaic}[1]{1}{8}{{e,d,a,c,c,b,b,a}}{{d}}{{i,i,e,f,h,h,g,g}}{{f}} 
\tileix \& \tilex \& \tileix\& \tilex \& \tileix \& \tilex \& \tileix \& \tilex \\
\end{vmosaic}\\
$8_4$ & $8_5$
\end{tabular}
}
\vspace{.2 in}

\medskip\noindent
\resizebox{\textwidth}{!}{%
\begin{tabular}{cc}
\begin{vmosaic}[1]{1}{8}{{e,d,d,c,b,b,a,a}}{{c}}{{e,f,h,i,i,h,g,g}}{{f}} 
\tileix \& \tilex \& \tileix\& \tilex \& \tileix \& \tilex \& \tileix \& \tilex \\
\end{vmosaic} &
\begin{vmosaic}[1]{1}{8}{{e,d,c,c,a,b,b,a}}{{d}}{{e,f,g,i,i,h,h,g}}{{f}} 
\tileix \& \tilex \& \tileix \& \tilex \& \tileix \& \tilex \& \tileix \& \tilex \\
\end{vmosaic}\\
$8_6$ & $8_7$
\end{tabular}
}
\vspace{.2 in}

\medskip\noindent
\resizebox{\textwidth}{!}{%
\begin{tabular}{cc}
\begin{vmosaic}[1]{1}{8}{{e,e,d,c,a,b,b,a}}{{c}}{{d,f,g,i,i,h,h,g}}{{f}} 
\tilex \& \tileix \& \tilex \& \tileix \& \tilex \& \tileix \& \tilex \& \tileix \\
\end{vmosaic} &
\begin{vmosaic}[1]{1}{8}{{e,e,d,c,a,b,b,a}}{{c}}{{d,f,i,i,g,h,h,g}}{{f}} 
\tilex \& \tileix \& \tilex \& \tileix \& \tilex \& \tileix \& \tilex \& \tileix \\
\end{vmosaic}\\
$8_8$ & $8_9$
\end{tabular}
}
\vspace{.2 in}

\medskip\noindent
\resizebox{\textwidth}{!}{%
\begin{tabular}{cc}
\begin{vmosaic}[1]{1}{8}{{e,d,b,c,c,b,a,a}}{{d}}{{e,i,i,f,g,h,h,g}}{{f}} 
\tileix \& \tilex \& \tileix \& \tilex \& \tileix \& \tilex \& \tileix \& \tilex \\
\end{vmosaic} &
\begin{vmosaic}[1]{1}{8}{{d,e,e,d,c,b,a,a}}{{b}}{{c,f,h,i,i,h,g,g}}{{f}} 
\tileix \& \tilex \& \tileix \& \tilex \& \tileix \& \tilex \& \tileix \& \tilex \\
\end{vmosaic}\\
$8_{10}$ & $8_{11}$
\end{tabular}
}
\vspace{.2 in}

\medskip\noindent
\resizebox{\textwidth}{!}{%
\begin{tabular}{cc}
\begin{vmosaic}[1]{1}{8}{{e,c,d,d,c,b,a,a}}{{b}}{{e,h,i,i,h,f,g,g}}{{f}} 
\tileix \& \tilex \& \tileix \& \tilex \& \tileix \& \tilex \& \tileix \& \tilex \\
\end{vmosaic} &
\begin{vmosaic}[1]{1}{8}{{e,d,a,b,c,c,b,a}}{{d}}{{e,i,i,f,g,h,h,g}}{{f}} 
\tilex \& \tileix \& \tilex \& \tileix \& \tilex \& \tileix \& \tilex \& \tileix \\
\end{vmosaic}\\
$8_{12}$ & $8_{13}$
\end{tabular}
}
\vspace{.2 in}

\medskip\noindent
\resizebox{\textwidth}{!}{%
\begin{tabular}{cc}
\begin{vmosaic}[1]{1}{8}{{e,d,a,b,c,c,b,a}}{{d}}{{i,i,e,f,g,h,h,g}}{{f}} 
\tilex \& \tileix \& \tilex \& \tileix \& \tilex \& \tileix \& \tilex \& \tileix \\
\end{vmosaic} &
\begin{vmosaic}[1]{1}{8}{{d,e,e,d,c,b,a,a}}{{b}}{{f,g,h,i,i,h,g,f}}{{c}} 
\tilex \& \tileix \& \tilex \& \tileix \& \tilex \& \tileix \& \tilex \& \tileix \\
\end{vmosaic}\\
$8_{14}$ & $8_{15}$
\end{tabular}
}
\vspace{.2 in}

\medskip\noindent
\resizebox{\textwidth}{!}{%
\begin{tabular}{cc}
\begin{vmosaic}[1]{1}{10}{{f,f,e,d,b,c,c,b,a,a}}{{d}}{{g,k,k,h,i,j,j,i,h,g}}{{e}} 
\tilex \& \tileix \& \tilex \& \tilex \& \tileix \& \tilex \& \tileix \& \tileix \& \tilex \& \tileix \\
\end{vmosaic} &
\begin{vmosaic}[1]{1}{8}{{f,e,d,c,a,b,b,a}}{{c}}{{d,e,f,g,h,i,i,h}}{{g}} 
\tileix \& \tilex \& \tileix \& \tilex \& \tileix \& \tilex \& \tileix \& \tilex \\
\end{vmosaic}\\
$8_{16}$ & $8_{17}$
\end{tabular}
}
\vspace{.2 in}

\medskip\noindent
\resizebox{\textwidth}{!}{%
\begin{tabular}{cc}
\begin{vmosaic}[1]{1}{10}{{f,e,c,d,d,c,b,a,a,b}}{{e}}{{f,g,j,k,k,j,h,i,i,h}}{{g}} 
\tileix \& \tilex \& \tileix \& \tileix \& \tilex \& \tileix \& \tilex \& \tilex \& \tileix \& \tilex\\
\end{vmosaic} &
\begin{vmosaic}[1]{1}{8}{{e,d,a,c,c,b,b,a}}{{d}}{{i,i,e,f,h,h,g,g}}{{f}} 
\tileix \& \tilex \& \tilex \& \tileix \& \tilex \& \tileix \& \tilex \& \tileix \\
\end{vmosaic}\\
$8_{18}$ & $8_{19}$
\end{tabular}
}
\vspace{.2 in}

\medskip\noindent
\resizebox{\textwidth}{!}{%
\begin{tabular}{cc}
\begin{vmosaic}[1]{1}{8}{{e,e,d,c,b,a,a,b}}{{c}}{{f,g,h,i,i,h,g,f}}{{d}} 
\tilex \& \tileix \& \tilex \& \tilex \& \tileix \& \tilex \& \tileix \& \tilex \& \\
\end{vmosaic} &
\begin{vmosaic}[1]{1}{8}{{e,e,d,c,a,b,b,a}}{{c}}{{f,g,h,i,i,h,g,f}}{{d}} 
\tileix \& \tilex \& \tileix \& \tilex \& \tileix \& \tileix \& \tilex \& \tilex\\
\end{vmosaic}\\
$8_{20}$ & $8_{21}$
\end{tabular}
}
\vspace{.2 in}

\section{Table of Alternating Knots that can be presented on $2 \times 3$ and $2 \times 4$ Virtual Mosaics}

\noindent
\resizebox{\textwidth}{!}{%
\begin{tabular}{ccc}
\begin{vmosaic}[1]{2}{3}{{c,b,a}}{{a,b}}{{c,d,e}}{{e,d}} 
\tileix \& \tilex \& \tilei\\
\tileiii \& \tileviii  \& \tilex\\
\end{vmosaic} &
\begin{vmosaic}[1]{2}{3}{{a,b,b}}{{a,d}}{{e,e,d}}{{c,c}} 
\tileix \& \tilex \& \tileix \\
\tilex \& \tileiv \& \tileiii\\
\end{vmosaic} &
\begin{vmosaic}[1]{2}{4}{{d,f,f,e}}{{e,d}}{{c,c,b,b}}{{a,a}} 
\tilex \& \tileviii \& \tilevii \& \tilei\\
\tileix \& \tilex \& \tileix \& \tilex\\
\end{vmosaic} \\
$3_1$ & $4_1$ & $5_1$
\end{tabular}
}
\vspace{.2 in}

\noindent
\resizebox{\textwidth}{!}{%
\begin{tabular}{ccc}
\begin{vmosaic}[1]{2}{3}{{c,b,b}}{{a,a}}{{c,e,e}}{{d,d}} 
\tileix \& \tileviii \& \tileix\\
\tilex \& \tileix  \& \tilex\\
\end{vmosaic} &
\begin{vmosaic}[1]{2}{4}{{a,b,b,a}}{{f,f}}{{c,d,e,e}}{{d,c}} 
\tilevii \& \tilex \& \tileix \& \tilex\\
\tilex \& \tileix \& \tileviii \& \tileix\\
\end{vmosaic} &
\begin{vmosaic}[1]{2}{4}{{c,c,b,a}}{{f,f}}{{e,e,d,d}}{{a,b}} 
\tileix \& \tilex \& \tileix \& \tilex\\
\tileiv \& \tileiii \& \tilex \& \tileix\\
\end{vmosaic} \\
$5_2$ & $6_1$ & $6_2$ 
\end{tabular}
}

\vspace{.2 in}

\noindent
\resizebox{\textwidth}{!}{%
\begin{tabular}{ccc}
\begin{vmosaic}[1]{2}{3}{{c,c,b}}{{a,d}}{{e,e,d}}{{a,b}} 
\tileix \& \tilex \& \tileix\\
\tilex \& \tileix  \& \tilex\\
\end{vmosaic} &
\begin{vmosaic}[1]{2}{4}{{c,b,a,a}}{{f,f}}{{b,c,e,e}}{{d,d}} 
\tileix \& \tileviii \& \tileix \& \tilex\\
\tilex \& \tileix \& \tilex \& \tileix\\
\end{vmosaic} &
\begin{vmosaic}[1]{2}{4}{{b,b,a,a}}{{d,f}}{{f,e,e,d}}{{c,c}} 
\tileix \& \tilex \& \tileix \& \tilex\\
\tilex \& \tileix \& \tilex \& \tileiv\\
\end{vmosaic} \\
$6_3$ & $7_4$ & $7_5$ 
\end{tabular}
}

\vspace{.2 in}

\noindent
\resizebox{\textwidth}{!}{%
\begin{tabular}{ccc}
\begin{vmosaic}[1]{2}{4}{{c,b,a,a}}{{b,c}}{{e,f,f,e}}{{d,d}} 
\tileix \& \tilex \& \tileix \& \tilex\\
\tilex \& \tileix  \& \tilex \& \tilevii\\
\end{vmosaic} &
\begin{vmosaic}[1]{2}{4}{{c,b,a,a}}{{b,c}}{{d,e,f,f}}{{e,d}} 
\tileix \& \tilex \& \tileix \& \tilex\\
\tilex \& \tileix \& \tilex \& \tilevii\\
\end{vmosaic} &
\begin{vmosaic}[1]{2}{4}{{c,b,a,a}}{{b,c}}{{e,f,f,e}}{{d,d}} 
\tileix \& \tilex \& \tileix \& \tilex\\
\tilex \& \tileix \& \tilex \& \tileix\\
\end{vmosaic} \\
$7_6$ & $7_7$ & $8_{15}$ 
\end{tabular}
}

\section{Table of Row Mosaics for Virtual Knots up to 4 Crossings}

\noindent
\resizebox{\textwidth}{!}{%
\begin{tabular}{cccc}
\begin{vmosaic}[1]{1}{2}{{a,b}}{{a}}{{c,b}}{{c}} 
\tileix \& \tileix \\
\end{vmosaic} &
\begin{vmosaic}[1]{1}{3}{{c,b,a}}{{c}}{{d,a,d}}{{b}} 
\tileix \& \tilex \& \tileix \\
\end{vmosaic} &
\begin{vmosaic}[1]{1}{3}{{c,b,a}}{{c}}{{d,d,a}}{{b}} 
\tilex \& \tilex \& \tileix \\
\end{vmosaic} &
\begin{vmosaic}[1]{1}{3}{{c,b,a}}{{c}}{{d,a,d}}{{b}} 
\tileix \& \tilex \& \tilex\\
\end{vmosaic} \\
$2.1,g=1$ & $3.1, g=2$ & $3.2, g=1$ & $3.3, g=2$ 
\end{tabular}
}
\vspace{.2 in}

\noindent
\resizebox{\textwidth}{!}{%
\begin{tabular}{cccc}
\begin{vmosaic}[1]{1}{3}{{c,b,a}}{{d}}{{c,a,d}}{{b}} 
\tileix \& \tilex \& \tilex\\
\end{vmosaic} &
\begin{vmosaic}[1]{1}{3}{{c,b,a}}{{c}}{{d,a,b}}{{d}} 
\tileix \& \tileix \& \tileix \\
\end{vmosaic} &
\begin{vmosaic}[1]{1}{3}{{b,a,a}}{{b}}{{c,d,d}}{{c}} 
\tileix \& \tilex \& \tileix \\
\end{vmosaic} &
\begin{vmosaic}[1]{1}{3}{{c,b,a}}{{c}}{{b,d,d}}{{a}} 
\tileix \& \tilex \& \tileix \\
\end{vmosaic} \\
$3.4, g=2$ & $3.5,g=1$ & $3.6, g=0$ & $3.7, g=1$
\end{tabular}
}

\vspace{.2 in}

\noindent
\resizebox{\textwidth}{!}{%
\begin{tabular}{ccc}
\begin{vmosaic}[1]{1}{4}{{d,c,a,b}}{{a}}{{e,b,d,e}}{{c}} 
\tilex \& \tilex \& \tileix \& \tileix\\
\end{vmosaic} &
\begin{vmosaic}[1]{1}{4}{{e,d,b,a}}{{c}}{{b,c,e,a}}{{d}} 
\tilex \& \tilex \& \tileix \& \tileix\\
\end{vmosaic} &
\begin{vmosaic}[1]{1}{4}{{d,c,b,a}}{{b}}{{e,e,d,a}}{{c}} 
\tilex \& \tilex \& \tileix \& \tilex\\
\end{vmosaic} \\
$4.1 , g=2$ & $4.2 ,g=2$ & $4.3 , g=2$ 
\end{tabular}
}
\vspace{.2 in}

\noindent
\resizebox{\textwidth}{!}{%
\begin{tabular}{ccc}
\begin{vmosaic}[1]{1}{4}{{d,e,a,b}}{{a}}{{c,b,d,c}}{{e}} 
\tilex \& \tilex \& \tileix \& \tilex\\
\end{vmosaic} &
\begin{vmosaic}[1]{1}{4}{{d,c,b,a}}{{b}}{{e,a,d,e}}{{c}} 
\tilex \& \tilex \& \tileix \& \tilex\\
\end{vmosaic} &
\begin{vmosaic}[1]{1}{4}{{d,e,b,b}}{{a}}{{c,a,d,c}}{{e}} 
\tilex \& \tilex \& \tileix \& \tilex\\
\end{vmosaic} \\
$4.4 , g=2$ & $4.5 ,g=2$ & $4.6 , g=2$ 
\end{tabular}
}
\vspace{.2 in}

\noindent
\resizebox{\textwidth}{!}{%
\begin{tabular}{ccc}
\begin{vmosaic}[1]{1}{4}{{d,c,b,a}}{{e}}{{b,e,d,a}}{{c}} 
\tilex \& \tilex \& \tilex \& \tilex\\
\end{vmosaic} &
\begin{vmosaic}[1]{1}{4}{{d,c,b,a}}{{b}}{{e,a,d,e}}{{c}} 
\tilex \& \tilex \& \tilex \& \tilex\\
\end{vmosaic} &
\begin{vmosaic}[1]{1}{4}{{d,c,b,a}}{{c}}{{e,d,a,e}}{{b}} 
\tilex \& \tileix \& \tilex \& \tileix\\
\end{vmosaic} \\
$4.7 , g=2$ & $4.8 ,g=2$ & $4.9 , g=2$ 
\end{tabular}
}
\vspace{.2 in}

\noindent
\resizebox{\textwidth}{!}{%
\begin{tabular}{ccc}
\begin{vmosaic}[1]{1}{4}{{d,c,b,a}}{{c}}{{e,d,e,a}}{{b}} 
\tilex \& \tileix \& \tilex \& \tileix\\
\end{vmosaic} &
\begin{vmosaic}[1]{1}{4}{{c,a,b,a}}{{e}}{{d,c,e,d}}{{b}} 
\tilex \& \tileix \& \tilex \& \tileix\\
\end{vmosaic} &
\begin{vmosaic}[1]{1}{4}{{d,c,b,a}}{{e}}{{c,d,e,a}}{{b}} 
\tilex \& \tileix \& \tilex \& \tileix\\
\end{vmosaic} \\
$4.10 , g=2$ & $4.11 ,g=2$ & $4.12 , g=1$ 
\end{tabular}
}
\vspace{.2 in}

\noindent
\resizebox{\textwidth}{!}{%
\begin{tabular}{ccc}
\begin{vmosaic}[1]{1}{4}{{b,c,b,a}}{{c}}{{e,d,e,a}}{{d}} 
\tilex \& \tileix \& \tilex \& \tileix\\
\end{vmosaic} &
\begin{vmosaic}[1]{1}{4}{{d,b,d,b}}{{a}}{{c,e,a,c}}{{e}} 
\tilex \& \tileix \& \tilex \& \tileix\\
\end{vmosaic} &
\begin{vmosaic}[1]{1}{4}{{d,a,e,c}}{{a}}{{b,d,b,c}}{{e}} 
\tilex \& \tileix \& \tilex \& \tilex\\
\end{vmosaic} \\
$4.13 , g=2$ & $4.14 ,g=2$ & $4.15 , g=2$ 
\end{tabular}
}
\vspace{.2 in}

\noindent
\resizebox{\textwidth}{!}{%
\begin{tabular}{ccc}
\begin{vmosaic}[1]{1}{4}{{d,c,b,a}}{{c}}{{e,d,a,e}}{{b}} 
\tilex \& \tileix \& \tilex \& \tilex\\
\end{vmosaic} &
\begin{vmosaic}[1]{1}{4}{{c,a,b,a}}{{e}}{{d,c,e,d}}{{b}} 
\tilex \& \tileix \& \tilex \& \tilex\\
\end{vmosaic} &
\begin{vmosaic}[1]{1}{4}{{c,b,e,b}}{{a}}{{c,d,a,d}}{{e}} 
\tilex \& \tileix \& \tilex \& \tilex\\
\end{vmosaic} \\
$4.16 , g=2$ & $4.17 ,g=2$ & $4.18 , g=2$ 
\end{tabular}
}
\vspace{.2 in}

\noindent
\resizebox{\textwidth}{!}{%
\begin{tabular}{ccc}
\begin{vmosaic}[1]{1}{4}{{b,c,b,a}}{{c}}{{e,d,a,e}}{{d}} 
\tilex \& \tileix \& \tilex \& \tilex\\
\end{vmosaic} &
\begin{vmosaic}[1]{1}{4}{{b,c,b,a}}{{e}}{{c,d,e,a}}{{d}} 
\tilex \& \tileix \& \tilex \& \tilex\\
\end{vmosaic} &
\begin{vmosaic}[1]{1}{4}{{b,a,b,a}}{{e}}{{d,c,e,d}}{{c}} 
\tilex \& \tileix \& \tilex \& \tilex\\
\end{vmosaic} \\
$4.19 , g=2$ & $4.20 ,g=2$ & $4.21 , g=2$ 
\end{tabular}
}
\vspace{.2 in}

\noindent
\resizebox{\textwidth}{!}{%
\begin{tabular}{ccc}
\begin{vmosaic}[1]{1}{4}{{d,c,b,a}}{{d}}{{b,e,a,e}}{{c}} 
\tileix \& \tilex \& \tileix \& \tileix\\
\end{vmosaic} &
\begin{vmosaic}[1]{1}{4}{{b,c,b,a}}{{d}}{{e,a,e,d}}{{c}} 
\tileix \& \tilex \& \tileix \& \tileix\\
\end{vmosaic} &
\begin{vmosaic}[1]{1}{4}{{d,c,b,a}}{{e}}{{b,d,a,e}}{{c}} 
\tileix \& \tilex \& \tileix \& \tileix\\
\end{vmosaic} \\
$4.22 , g=2$ & $4.23 ,g=2$ & $4.24 , g=2$ 
\end{tabular}
}
\vspace{.2 in}

\noindent
\resizebox{\textwidth}{!}{%
\begin{tabular}{ccc}
\begin{vmosaic}[1]{1}{4}{{d,c,b,a}}{{d}}{{b,e,e,a}}{{c}} 
\tileix \& \tilex \& \tileix \& \tileix\\
\end{vmosaic} &
\begin{vmosaic}[1]{1}{4}{{d,c,b,a}}{{d}}{{e,a,b,e}}{{c}} 
\tileix \& \tilex \& \tileix \& \tileix\\
\end{vmosaic} &
\begin{vmosaic}[1]{1}{4}{{a,c,b,a}}{{d}}{{b,e,e,d}}{{c}} 
\tileix \& \tilex \& \tileix \& \tileix\\
\end{vmosaic} \\
$4.25 , g=2$ & $4.26 ,g=2$ & $4.27 , g=2$ 
\end{tabular}
}
\vspace{.2 in}

\noindent
\resizebox{\textwidth}{!}{%
\begin{tabular}{ccc}
\begin{vmosaic}[1]{1}{4}{{d,c,b,a}}{{e}}{{d,a,b,e}}{{c}} 
\tileix \& \tilex \& \tileix \& \tileix\\
\end{vmosaic} &
\begin{vmosaic}[1]{1}{4}{{d,c,b,a}}{{b}}{{e,d,e,a}}{{c}} 
\tileix \& \tilex \& \tileix \& \tileix\\
\end{vmosaic} &
\begin{vmosaic}[1]{1}{4}{{d,c,b,a}}{{b}}{{e,d,a,e}}{{c}} 
\tileix \& \tilex \& \tileix \& \tileix\\
\end{vmosaic} \\
$4.28 , g=2$ & $4.29 ,g=2$ & $4.30 , g=2$ 
\end{tabular}
}
\vspace{.2 in}

\noindent
\resizebox{\textwidth}{!}{%
\begin{tabular}{ccc}
\begin{vmosaic}[1]{1}{4}{{b,c,b,a}}{{e}}{{d,e,d,a}}{{c}} 
\tileix \& \tilex \& \tileix \& \tileix\\
\end{vmosaic} &
\begin{vmosaic}[1]{1}{4}{{b,c,b,a}}{{e}}{{d,e,a,d}}{{c}} 
\tileix \& \tilex \& \tileix \& \tileix\\
\end{vmosaic} &
\begin{vmosaic}[1]{1}{4}{{b,c,a,b}}{{a}}{{d,e,d,e}}{{c}} 
\tileix \& \tilex \& \tileix \& \tileix\\
\end{vmosaic} \\
$4.31 , g=2$ & $4.32 ,g=2$ & $4.33 , g=2$ 
\end{tabular}
}
\vspace{.2 in}

\noindent
\resizebox{\textwidth}{!}{%
\begin{tabular}{ccc}
\begin{vmosaic}[1]{1}{4}{{d,c,b,a}}{{b}}{{d,e,a,e}}{{c}} 
\tileix \& \tilex \& \tileix \& \tileix\\
\end{vmosaic} &
\begin{vmosaic}[1]{1}{4}{{a,c,b,a}}{{e}}{{d,e,d,b}}{{c}} 
\tileix \& \tilex \& \tileix \& \tileix\\
\end{vmosaic} &
\begin{vmosaic}[1]{1}{4}{{d,c,b,a}}{{e}}{{d,e,a,b}}{{c}} 
\tileix \& \tilex \& \tileix \& \tileix\\
\end{vmosaic} \\
$4.34 , g=2$ & $4.35 ,g=2$ & $4.36 , g=1$ 
\end{tabular}
}
\vspace{.2 in}

\noindent
\resizebox{\textwidth}{!}{%
\begin{tabular}{ccc}
\begin{vmosaic}[1]{1}{4}{{d,c,b,a}}{{d}}{{e,e,a,b}}{{c}} 
\tileix \& \tilex \& \tileix \& \tilex\\
\end{vmosaic} &
\begin{vmosaic}[1]{1}{4}{{d,c,b,a}}{{d}}{{e,a,e,b}}{{c}} 
\tileix \& \tilex \& \tileix \& \tilex\\
\end{vmosaic} &
\begin{vmosaic}[1]{1}{4}{{d,c,b,a}}{{d}}{{b,e,a,e}}{{c}} 
\tileix \& \tilex \& \tileix \& \tilex\\
\end{vmosaic} \\
$4.37 , g=1$ & $4.38 ,g=2$ & $4.39 , g=2$ 
\end{tabular}
}
\vspace{.2 in}

\noindent
\resizebox{\textwidth}{!}{%
\begin{tabular}{ccc}
\begin{vmosaic}[1]{1}{4}{{c,b,a,a}}{{c}}{{e,d,e,d}}{{b}} 
\tileix \& \tilex \& \tileix \& \tilex\\
\end{vmosaic} &
\begin{vmosaic}[1]{1}{4}{{b,c,b,a}}{{d}}{{e,e,a,d}}{{c}} 
\tileix \& \tilex \& \tileix \& \tilex\\
\end{vmosaic} &
\begin{vmosaic}[1]{1}{4}{{d,c,b,a}}{{e}}{{b,d,a,e}}{{c}} 
\tileix \& \tilex \& \tileix \& \tilex\\
\end{vmosaic} \\
$4.40 , g=2$ & $4.41 ,g=2$ & $4.42 , g=2$ 
\end{tabular}
}
\vspace{.2 in}

\noindent
\resizebox{\textwidth}{!}{%
\begin{tabular}{ccc}
\begin{vmosaic}[1]{1}{4}{{c,b,a,a}}{{c}}{{d,e,e,d}}{{b}} 
\tileix \& \tilex \& \tileix \& \tilex\\
\end{vmosaic} &
\begin{vmosaic}[1]{1}{4}{{d,c,b,a}}{{d}}{{b,e,e,a}}{{c}} 
\tileix \& \tilex \& \tileix \& \tilex\\
\end{vmosaic} &
\begin{vmosaic}[1]{1}{4}{{d,c,b,a}}{{d}}{{e,a,b,e}}{{c}} 
\tileix \& \tilex \& \tileix \& \tilex\\
\end{vmosaic} \\
$4.43 , g=1$ & $4.44 ,g=2$ & $4.45 , g=2$ 
\end{tabular}
}

\vspace{.2 in}

\noindent
\resizebox{\textwidth}{!}{%
\begin{tabular}{ccc}
\begin{vmosaic}[1]{1}{4}{{d,c,b,a}}{{d}}{{e,e,b,a}}{{c}} 
\tileix \& \tilex \& \tileix \& \tilex\\
\end{vmosaic} &
\begin{vmosaic}[1]{1}{4}{{d,c,b,a}}{{e}}{{d,a,b,e}}{{c}} 
\tileix \& \tilex \& \tileix \& \tilex\\
\end{vmosaic} &
\begin{vmosaic}[1]{1}{4}{{d,c,b,a}}{{b}}{{e,d,a,e}}{{c}} 
\tileix \& \tilex \& \tileix \& \tilex\\
\end{vmosaic} \\
$4.46 , g=2$ & $4.47 ,g=2$ & $4.48 , g=2$ 
\end{tabular}
}
\vspace{.2 in}

\noindent
\resizebox{\textwidth}{!}{%
\begin{tabular}{ccc}
\begin{vmosaic}[1]{1}{4}{{d,c,b,a}}{{b}}{{e,d,e,a}}{{c}} 
\tileix \& \tilex \& \tileix \& \tilex\\
\end{vmosaic} &
\begin{vmosaic}[1]{1}{4}{{b,c,b,a}}{{e}}{{d,e,a,d}}{{c}} 
\tileix \& \tilex \& \tileix \& \tilex\\
\end{vmosaic} &
\begin{vmosaic}[1]{1}{4}{{b,c,b,a}}{{e}}{{d,e,d,a}}{{c}} 
\tileix \& \tilex \& \tileix \& \tilex\\
\end{vmosaic} \\
$4.49 , g=2$ & $4.50 ,g=2$ & $4.51 , g=2$ 
\end{tabular}
}
\vspace{.2 in}

\noindent
\resizebox{\textwidth}{!}{%
\begin{tabular}{ccc}
\begin{vmosaic}[1]{1}{4}{{a,c,b,a}}{{b}}{{e,d,e,d}}{{c}} 
\tileix \& \tilex \& \tileix \& \tilex\\
\end{vmosaic} &
\begin{vmosaic}[1]{1}{4}{{a,c,b,a}}{{b}}{{e,e,d,d}}{{c}} 
\tileix \& \tilex \& \tileix \& \tilex\\
\end{vmosaic} &
\begin{vmosaic}[1]{1}{4}{{d,c,b,a}}{{b}}{{d,a,e,e}}{{c}} 
\tileix \& \tilex \& \tileix \& \tilex\\
\end{vmosaic} \\
$4.52 , g=2$ & $4.53 ,g=1$ & $4.54 , g=2$ 
\end{tabular}
}
 \vspace{.2 in}

\noindent
\resizebox{\textwidth}{!}{%
\begin{tabular}{ccc}
\begin{vmosaic}[1]{1}{4}{{d,c,b,a}}{{b}}{{e,a,d,e}}{{c}} 
\tileix \& \tilex \& \tileix \& \tilex\\
\end{vmosaic} &
\begin{vmosaic}[1]{1}{4}{{d,c,b,a}}{{e}}{{b,e,d,a}}{{c}} 
\tileix \& \tilex \& \tileix \& \tilex\\
\end{vmosaic} &
\begin{vmosaic}[1]{1}{4}{{c,a,b,a}}{{e}}{{d,c,e,d}}{{b}} 
\tilex \& \tilex \& \tilex \& \tilex\\
\end{vmosaic} \\
$4.55 , g=2$ & $4.56 ,g=2$ & $4.57 , g=2$ 
\end{tabular}
}
\vspace{.2 in}

\noindent
\resizebox{\textwidth}{!}{%
\begin{tabular}{ccc}
\begin{vmosaic}[1]{1}{4}{{d,c,b,a}}{{c}}{{e,d,a,e}}{{b}} 
\tilex \& \tilex \& \tilex \& \tilex\\
\end{vmosaic} &
\begin{vmosaic}[1]{1}{4}{{b,a,b,a}}{{e}}{{d,c,e,d}}{{c}} 
\tilex \& \tilex \& \tilex \& \tilex\\
\end{vmosaic} &
\begin{vmosaic}[1]{1}{4}{{b,c,b,a}}{{c}}{{e,d,e,a}}{{d}} 
\tilex \& \tilex \& \tilex \& \tilex\\
\end{vmosaic} \\
$4.58 , g=2$ & $4.59 ,g=2$ & $4.60 , g=2$ 
\end{tabular}
}
\vspace{.2 in}

\noindent
\resizebox{\textwidth}{!}{%
\begin{tabular}{ccc}
\begin{vmosaic}[1]{1}{4}{{c,b,b,a}}{{e}}{{c,d,e,d}}{{a}} 
\tilex \& \tilex \& \tileix \& \tilex\\
\end{vmosaic} &
\begin{vmosaic}[1]{1}{4}{{d,c,b,a}}{{e}}{{d,c,e,b}}{{a}} 
\tilex \& \tilex \& \tileix \& \tilex\\
\end{vmosaic} &
\begin{vmosaic}[1]{1}{4}{{c,b,d,a}}{{c}}{{e,a,e,d}}{{b}} 
\tileix \& \tilex \& \tilex \& \tileix\\
\end{vmosaic} \\
$4.61 , g=2$ & $4.62 ,g=2$ & $4.63 , g=2$ 
\end{tabular}
}
\vspace{.2 in}

\noindent
\resizebox{\textwidth}{!}{%
\begin{tabular}{ccc}
\begin{vmosaic}[1]{1}{4}{{d,c,b,a}}{{c}}{{d,e,e,b}}{{a}} 
\tilex \& \tilex \& \tileix \& \tilex\\
\end{vmosaic} &
\begin{vmosaic}[1]{1}{4}{{c,b,a,a}}{{d}}{{e,c,e,d}}{{b}} 
\tileix \& \tilex \& \tilex \& \tileix\\
\end{vmosaic} &
\begin{vmosaic}[1]{1}{4}{{d,c,b,a}}{{e}}{{b,d,a,e}}{{c}} 
\tileix \& \tilex \& \tilex \& \tileix\\
\end{vmosaic} \\
$4.64 , g=1$ & $4.65 ,g=1$ & $4.66 , g=2$ 
\end{tabular}
}
\vspace{.2 in}

\noindent
\resizebox{\textwidth}{!}{%
\begin{tabular}{ccc}
\begin{vmosaic}[1]{1}{4}{{b,c,b,a}}{{d}}{{e,a,e,d}}{{c}} 
\tileix \& \tilex \& \tilex \& \tileix\\
\end{vmosaic} &
\begin{vmosaic}[1]{1}{4}{{b,c,b,a}}{{d}}{{e,e,a,d}}{{c}} 
\tileix \& \tilex \& \tilex \& \tileix\\
\end{vmosaic} &
\begin{vmosaic}[1]{1}{4}{{b,c,b,a}}{{e}}{{d,e,d,a}}{{c}} 
\tileix \& \tilex \& \tilex \& \tileix\\
\end{vmosaic} \\
$4.67 , g=2$ & $4.68 ,g=2$ & $4.69 , g=2$ 
\end{tabular}
}
\vspace{.2 in}

\noindent
\resizebox{\textwidth}{!}{%
\begin{tabular}{ccc}
\begin{vmosaic}[1]{1}{4}{{d,c,b,a}}{{b}}{{e,d,e,a}}{{c}} 
\tileix \& \tilex \& \tilex \& \tileix\\
\end{vmosaic} &
\begin{vmosaic}[1]{1}{4}{{d,c,b,a}}{{b}}{{e,d,a,e}}{{c}} 
\tileix \& \tilex \& \tilex \& \tileix\\
\end{vmosaic} &
\begin{vmosaic}[1]{1}{4}{{b,c,a,b}}{{a}}{{e,d,e,d}}{{c}} 
\tileix \& \tilex \& \tilex \& \tileix\\
\end{vmosaic} \\
$4.70 , g=2$ & $4.71 ,g=2$ & $4.72 , g=2$ 
\end{tabular}
}
\vspace{.2 in}

\noindent
\resizebox{\textwidth}{!}{%
\begin{tabular}{ccc}
\begin{vmosaic}[1]{1}{4}{{c,b,a,a}}{{e}}{{c,e,d,d}}{{b}} 
\tileix \& \tilex \& \tilex \& \tileix\\
\end{vmosaic} &
\begin{vmosaic}[1]{1}{4}{{a,c,b,a}}{{e}}{{b,e,d,d}}{{c}} 
\tileix \& \tilex \& \tilex \& \tileix\\
\end{vmosaic} &
\begin{vmosaic}[1]{1}{4}{{a,c,b,a}}{{b}}{{e,e,d,d}}{{c}} 
\tileix \& \tilex \& \tilex \& \tileix\\
\end{vmosaic} \\
$4.73 , g=1$ & $4.74 ,g=2$ & $4.75 , g=1$ 
\end{tabular}
}
\vspace{.2 in}

\noindent
\resizebox{\textwidth}{!}{%
\begin{tabular}{ccc}
\begin{vmosaic}[1]{1}{4}{{d,c,b,a}}{{e}}{{b,e,d,a}}{{c}} 
\tileix \& \tilex \& \tilex \& \tileix\\
\end{vmosaic} &
\begin{vmosaic}[1]{1}{4}{{d,c,b,a}}{{b}}{{e,a,d,e}}{{c}} 
\tileix \& \tilex \& \tilex \& \tileix\\
\end{vmosaic} &
\begin{vmosaic}[1]{1}{4}{{d,c,b,a}}{{d}}{{b,e,a,e}}{{c}} 
\tileix \& \tilex \& \tilex \& \tilex\\
\end{vmosaic} \\
$4.76 , g=2$ & $4.77 ,g=2$ & $4.78 , g=2$ 
\end{tabular}
}
\vspace{.2 in}

\noindent
\resizebox{\textwidth}{!}{%
\begin{tabular}{ccc}
\begin{vmosaic}[1]{1}{4}{{d,c,b,a}}{{e}}{{b,d,a,e}}{{c}} 
\tileix \& \tilex \& \tilex \& \tilex\\
\end{vmosaic} &
\begin{vmosaic}[1]{1}{4}{{d,c,b,a}}{{d}}{{e,a,b,e}}{{c}} 
\tileix \& \tilex \& \tilex \& \tilex\\
\end{vmosaic} &
\begin{vmosaic}[1]{1}{4}{{d,c,b,a}}{{e}}{{d,a,b,e}}{{c}} 
\tileix \& \tilex \& \tilex \& \tilex\\
\end{vmosaic} \\
$4.79 , g=2$ & $4.80 ,g=2$ & $4.81 , g=2$ 
\end{tabular}
}
\vspace{.2 in}

\noindent
\resizebox{\textwidth}{!}{%
\begin{tabular}{ccc}
\begin{vmosaic}[1]{1}{4}{{d,c,b,a}}{{d}}{{e,e,a,c}}{{b}} 
\tileix \& \tileix \& \tilex \& \tileix\\
\end{vmosaic} &
\begin{vmosaic}[1]{1}{4}{{d,c,b,a}}{{d}}{{e,a,e,c}}{{b}} 
\tileix \& \tileix \& \tilex \& \tileix\\
\end{vmosaic} &
\begin{vmosaic}[1]{1}{4}{{c,a,b,a}}{{d}}{{e,e,c,d}}{{b}} 
\tileix \& \tileix \& \tilex \& \tileix\\
\end{vmosaic} \\
$4.82 , g=2$ & $4.83 ,g=2$ & $4.84 , g=2$ 
\end{tabular}
}
\vspace{.2 in}

\noindent
\resizebox{\textwidth}{!}{%
\begin{tabular}{ccc}
\begin{vmosaic}[1]{1}{4}{{d,c,b,a}}{{c}}{{e,a,d,e}}{{b}} 
\tileix \& \tileix \& \tilex \& \tileix\\
\end{vmosaic} &
\begin{vmosaic}[1]{1}{4}{{a,c,b,a}}{{d}}{{e,e,d,c}}{{b}} 
\tileix \& \tileix \& \tilex \& \tileix\\
\end{vmosaic} &
\begin{vmosaic}[1]{1}{4}{{d,c,b,a}}{{d}}{{e,a,e,c}}{{b}} 
\tileix \& \tileix \& \tilex \& \tilex\\
\end{vmosaic} \\
$4.85 , g=2$ & $4.86 ,g=1$ & $4.87 , g=2$ 
\end{tabular}
}
\vspace{.2 in}

\noindent
\resizebox{\textwidth}{!}{%
\begin{tabular}{ccc}
\begin{vmosaic}[1]{1}{4}{{d,c,b,a}}{{e}}{{c,a,d,e}}{{b}} 
\tileix \& \tileix \& \tilex \& \tilex\\
\end{vmosaic} &
\begin{vmosaic}[1]{1}{4}{{d,c,b,a}}{{c}}{{e,a,d,e}}{{b}} 
\tileix \& \tileix \& \tilex \& \tilex\\
\end{vmosaic} &
\begin{vmosaic}[1]{1}{4}{{d,c,b,a}}{{e}}{{d,a,e,c}}{{b}} 
\tileix \& \tileix \& \tilex \& \tilex\\
\end{vmosaic} \\
$4.88 , g=2$ & $4.89 ,g=2$ & $4.90 , g=2$ 
\end{tabular}
}
\vspace{.2 in}

\noindent
\resizebox{\textwidth}{!}{%
\begin{tabular}{ccc}
\begin{vmosaic}[1]{1}{4}{{d,c,b,a}}{{d}}{{e,a,b,c}}{{e}} 
\tileix \& \tileix \& \tileix \& \tileix\\
\end{vmosaic} &
\begin{vmosaic}[1]{1}{4}{{d,c,b,a}}{{d}}{{e,c,a,b}}{{e}} 
\tileix \& \tileix \& \tileix \& \tileix\\
\end{vmosaic} &
\begin{vmosaic}[1]{1}{4}{{d,c,b,a}}{{d}}{{e,c,e,b}}{{a}} 
\tileix \& \tileix \& \tileix \& \tileix\\
\end{vmosaic} \\
$4.91 , g=1$ & $4.92 ,g=1$ & $4.93 , g=2$ 
\end{tabular}
}
\vspace{.2 in}

\noindent
\resizebox{\textwidth}{!}{%
\begin{tabular}{ccc}
\begin{vmosaic}[1]{1}{4}{{c,a,b,a}}{{c}}{{e,e,d,d}}{{b}} 
\tileix \& \tilex \& \tilex \& \tileix\\
\end{vmosaic} &
\begin{vmosaic}[1]{1}{4}{{d,c,b,a}}{{d}}{{c,a,e,e}}{{b}} 
\tileix \& \tilex \& \tilex \& \tileix\\
\end{vmosaic} &
\begin{vmosaic}[1]{1}{4}{{d,c,b,a}}{{d}}{{e,e,a,c}}{{b}} 
\tileix \& \tilex \& \tilex \& \tileix\\
\end{vmosaic} \\
$4.94 , g=1$ & $4.95 ,g=1$ & $4.96 , g=2$ 
\end{tabular}
}
\vspace{.2 in}

\noindent
\resizebox{\textwidth}{!}{%
\begin{tabular}{ccc}
\begin{vmosaic}[1]{1}{4}{{d,c,b,a}}{{d}}{{e,a,e,c}}{{b}} 
\tileix \& \tilex \& \tilex \& \tileix\\
\end{vmosaic} &
\begin{vmosaic}[1]{1}{4}{{d,c,b,a}}{{c}}{{e,a,d,e}}{{b}} 
\tileix \& \tilex \& \tilex \& \tileix\\
\end{vmosaic} &
\begin{vmosaic}[1]{1}{4}{{a,c,b,a}}{{c}}{{e,e,d,d}}{{b}} 
\tileix\& \tilex \& \tilex \& \tileix\\
\end{vmosaic} \\
$4.97 , g=2$ & $4.98 ,g=2$ & $4.99 , g=1$ 
\end{tabular}
}
\vspace{.2 in}

\noindent
\resizebox{\textwidth}{!}{%
\begin{tabular}{ccc}
\begin{vmosaic}[1]{1}{4}{{c,b,a,a}}{{c}}{{d,e,e,b}}{{d}} 
\tileix \& \tileix \& \tilex \& \tileix\\
\end{vmosaic} &
\begin{vmosaic}[1]{1}{4}{{d,c,b,a}}{{d}}{{b,e,e,c}}{{a}} 
\tileix \& \tileix \& \tilex \& \tileix\\
\end{vmosaic} &
\begin{vmosaic}[1]{1}{4}{{d,c,b,a}}{{d}}{{e,e,b,c}}{{a}} 
\tileix \& \tileix \& \tilex \& \tileix\\
\end{vmosaic} \\
$4.100 , g=1$ & $4.101 ,g=1$ & $4.102 , g=1$ 
\end{tabular}
}
\vspace{.2 in}

\noindent
\resizebox{\textwidth}{!}{%
\begin{tabular}{ccc}
\begin{vmosaic}[1]{1}{4}{{d,c,b,a}}{{c}}{{b,e,d,e}}{{a}} 
\tileix \& \tileix \& \tilex \& \tileix\\
\end{vmosaic} &
\begin{vmosaic}[1]{1}{4}{{d,c,b,a}}{{c}}{{e,e,d,b}}{{a}} 
\tileix \& \tileix \& \tilex \& \tileix\\
\end{vmosaic} &
\begin{vmosaic}[1]{1}{4}{{a,c,b,a}}{{c}}{{e,e,d,d}}{{b}} 
\tilex \& \tileix \& \tilex \& \tileix\\
\end{vmosaic} \\
$4.103 , g=2$ & $4.104 ,g=1$ & $4.105 , g=1$ 
\end{tabular}
}
\vspace{.2 in}

\noindent
\resizebox{\textwidth}{!}{%
\begin{tabular}{ccc}
\begin{vmosaic}[1]{1}{4}{{c,e,b,a}}{{e}}{{c,a,d,d}}{{b}} 
\tilex \& \tileix \& \tilex \& \tileix\\
\end{vmosaic} &
\begin{vmosaic}[1]{1}{4}{{d,c,b,a}}{{e}}{{d,a,e,c}}{{b}} 
\tilex \& \tileix \& \tilex \& \tileix\\
\end{vmosaic} &
\begin{vmosaic}[1]{1}{4}{{c,c,b,a}}{{d}}{{e,e,d,a}}{{b}} 
\tilex \& \tileix \& \tilex \& \tileix\\
\end{vmosaic} \\
$4.106 , g=1$ & $4.107 ,g=2$ & $4.108 , g=0$ 
\end{tabular}
}

\bibliography{main}  
\bibliographystyle{alpha}

\end{document}